\documentclass[a4paper,11pt]{article}

\usepackage{geometry}
\usepackage{amsfonts}
\usepackage{amsthm}
\usepackage{amssymb}
\usepackage{amsmath}
\usepackage{upref}
\usepackage{mathrsfs}
\usepackage{color}
\usepackage[citecolor=blue,colorlinks=true]{hyperref}
\usepackage{enumitem}

\usepackage[english]{babel}
\usepackage{amsmath,amsthm}
\usepackage{geometry}
\usepackage{amssymb}
\usepackage{verbatim}
\usepackage{graphicx}
\usepackage{color}
\usepackage{xcolor}
\usepackage{mathrsfs}
\numberwithin{equation}{section}

\theoremstyle{plain}
\newtheorem{theorem}{Theorem}[section]
\newtheorem{proposition}[theorem]{Proposition}
\newtheorem{lemma}[theorem]{Lemma}
\newtheorem{corollary}[theorem]{Corollary}

\theoremstyle{definition}
\newtheorem{definition}[theorem]{Definition}

\theoremstyle{remark}
\newtheorem{remark}[theorem]{Remark}

\numberwithin{equation}{section}
\def\N{\mathbb{N}}
\def\Z{\mathbb{Z}}

\def\R{\mathbb{R}}

\def\Lip{\mathrm{Lip}}
\def\ds{\displaystyle} 
\def\div{{\rm div}}
\def\ocirc#1{\ifmmode\setbox0=\hbox{$#1$}\dimen0=\ht0
    \advance\dimen0 by1pt\rlap{\hbox to\wd0{\hss\raise\dimen0
    \hbox{\hskip.2em$\scriptscriptstyle\circ$}\hss}}#1\else
    {\accent"17 #1}\fi} 
\def\sgn{\mathrm{sgn}}
\def\eps{\varepsilon}
\def\<{\langle}
\def\>{\rangle}
\def\F{\mathcal{F}}
\def\G{\mathcal{G}}
\def\P{\mathbb{P}}
\def\E{\mathbb{E}}
\def\T{\mathbb{T}}
\def\M{M}
\def\QLB{Q_{\mathrm{LB}}}
\def\QFP{Q_{\mathrm{FP}}}
\def\LB{\mathrm{LB}}
\def\FP{\mathrm{FP}}
\def\e{\mathbf{e}}

\def\L{\mathscr{L}}

\newcommand{\osym}{\overset{\mathtt{sym}}{\otimes}}
\newcommand{\dual}[2]{\langle #1, #2\rangle}

\def\ps{{:}}

\def\RR{\mathtt{C}}

\begin{document}

\title{Diffusion-approximation  in stochastically forced kinetic equations}
\author{Arnaud Debussche\thanks{Univ Rennes, CNRS, IRMAR - UMR 6625, F-35000 Rennes, France. Email: arnaud.debussche@ens-rennes.fr, partially supported by the French government thanks to the the ``Investissements d'Avenir" program ANR-11-LABX-0020-01} and Julien Vovelle\thanks{Univ Lyon, CNRS, ENS Lyon, UMPA - UMR 5669, F-69364 Lyon, France, partially supported by the ANR project STAB and the ``Investissements d'Avenir" program LABEX MILYON ANR-10-LABX-0070}}
\maketitle

\begin{abstract} We derive the hydrodynamic limit of a kinetic equation where the interactions in velocity are modeled by a linear operator (Fokker-Planck or Linear Boltzmann) and the force in the Vlasov term is a stochastic process with high amplitude and short-range correlation. In the scales and the regime we consider, the hydrodynamic equation is a scalar second-order stochastic partial differential equation. Compared to the deterministic case, we also observe a phenomenon of enhanced diffusion.
\end{abstract}
{\bf Keywords:} diffusion-approximation, kinetic equation, hydrodynamic limit\medskip 

{\bf MSC Number: } 35Q20 (35R60 60H15 35B40)

\normalsize

\tableofcontents


\section{Introduction}

\subsection{Kinetic equations}

Let $N\in\N^*$. We denote by $\T^N$ the $N$-dimensional torus. Let $\eps>0$. We consider the following kinetic equation
\begin{equation}\label{Perturb}
\partial_t f+\eps v\cdot\nabla_x f+\bar E(t,x)\cdot\nabla_v f=Qf, \quad t>0,x\in\T^N,v\in\R^N,
\end{equation}
which is a perturbation of the equation
\begin{equation}\label{Stov}
\partial_t f+\bar E(t,x)\cdot\nabla_v f=Qf\quad t>0,x\in\T^N,v\in\R^N. 
\end{equation}
The operator $Q$ is either the linear Boltzmann (LB) operator
\begin{equation}\label{QLB}
\QLB f=\rho(f)\M-f,\quad\rho(f)=\int_{\R^N}f(v)dv,\quad\M(v)=\frac{1}{(2\pi)^{N/2}}\exp\left(-\ds\frac{|v|^2}{2}\right),
\end{equation}
or the Fokker-Planck (FP) operator
\begin{equation}\label{QFP}
\QFP f=\div_v(\nabla_v f+vf).
\end{equation}
The force field $\bar E(t,x)$ in \eqref{Stov} is a Markov, stationary mixing process $t\mapsto\bar{E}(t)$ with state space $F=H^m(\T^N;\R^N)$ ($m>N+1$). Mixing here refers to the mixing property as defined for stochastic processes (asymptotic independence), see Section~\ref{sec:MixE}. We show in Section~\ref{sec:unperturbed} that there is a unique, ergodic, invariant measure for \eqref{Stov} and that this invariant measure is the law of an invariant solution $(x,v)\mapsto\rho(x)\bar{\M}_t(x,v)$ parametrized by $\rho(x)$. See \eqref{MbarLB}-\eqref{MbarFP} for the definition of $\bar{\M}_t$. Consider the solution $f$ to \eqref{Perturb} starting from a state
\begin{equation}\label{approx0}
f_\mathrm{in}(x,v)\approx \rho_\mathrm{in}(x)\overline{\M}_0(x,v).
\end{equation}
Rescale over time intervals of order $\eps^{-2}$:
\begin{equation}\label{rescale}
f^\eps(t,x,v)=f(\eps^{-2}t,x,v).
\end{equation}
Then $f^\eps$ is a solution to the equation
\begin{equation}\label{Eqrescaled}
\partial_t f^\eps+\frac{v}{\eps}\cdot\nabla_x f^\eps+\frac{1}{\eps^2}\bar E(\eps^{-2}t,x)\cdot\nabla_v f^\eps=\frac{1}{\eps^2} Qf^\eps, \quad t>0,x\in\T^N,v\in\R^N.
\end{equation}
On bounded time intervals $[0,T]$, we expect
\begin{equation}\label{approxt}
f^\eps(t,x,v)\approx\rho(x,t)\overline{\M}_{\eps^{-2}t}(x,v),
\end{equation}
where $\rho$ is solution to a given equation (the \textit{hydrodynamic} equation) which we would like to identify. We do not prove \eqref{approxt}, but find the limit equation satisfied by $\rho=\lim_{\eps\to 0}\rho^\eps$, where $\rho^\eps=\rho(f^\eps)$. We show in Theorem~\ref{th:AD} that $\rho$ satisfies a diffusion equation, where the drift term is a second order differential operator in divergence form with respect to the space-variable $x$. Showing that $\rho^\eps$ is close to $\rho$ with $\rho$ a diffusion (in infinite dimension) is therefore a result of diffusion-approximation (in infinite dimension). See Theorem~\ref{th:AD} for the precise statement. 
%
%
\subsection{Trajectories}

The phase space associated to \eqref{Perturb} is $\T^N\times\R^N$. Consider the following systems of stochastic differential equations:
\begin{equation} \label{trajLB}
\begin{split}
dX_t& =\eps dV_t, \\
 dV_t& = \bar{E}(t,X_t)dt+\mbox{jumps},
\end{split}
\end{equation}
and
\begin{equation} \label{trajFP}
\begin{split}
dX_t& =\eps dV_t, \\
 dV_t& = (\bar{E}(t,X_t)-V_t) dt+\sqrt{2}dB_t.
\end{split}
\end{equation}
In \eqref{trajLB} the second equation describes the following piecewise deterministic Markov process (PDMP). Consider the Poisson process associated to the times $(T_n)$ and to the probability measure $\M dv$: the increments $T_{n+1}-T_n$ are i.i.d. with exponential law of parameter $1$. At each time $t=T_n$, $V_t$ is jumping to a new value $V_{T_n+}$ chosen at random, according to the probability law $\M dv$. Between each jump, $(V_t)$ is evolving by the differential equation
\begin{equation}\label{PDMP0}
\frac{d V_t}{dt}=E(t,X_t),\quad T_n<t<T_{n+1},
\end{equation}
which is coupled with the first equation of \eqref{trajLB}. In \eqref{trajFP}, $B_t$ is an $N$-dimensional Wiener process. In both the $\LB$ case and the $\FP$ case, the extra stochastic processes which we introduce are independent of $(\bar{E}(t))$. In this context, Equation~\eqref{Perturb} gives the evolution of the density, with respect to the Lebesgue measure on $\T^N_x\times\R^N_v$, of the conditional law of $(X_t,V_t)$: let $\F^E_t=\sigma((\bar{E}_s)_{0\leq s\leq t})$. If the law of $(X_0,V_0)$ has density $f_\mathrm{in}$ with respect to the Lebesgue measure on $\T^N_x\times\R^N_v$, then
\begin{equation}\label{ffromlaw}
\E\left[\varphi(X_t,V_t)|\F^E_t\right]=\iint_{\T^N\times\R^N} \varphi(x,v)f_t(x,v) dxdv,
\end{equation}
for all $\varphi\in C_b(\T^N\times\R^N)$. From \eqref{ffromlaw}, it follows that
\begin{equation}\label{rhofromlaw}
\E\left[\varphi(X_t)\right]=\int_{\T^N} \varphi(x)\E\rho_t(x) dx,\quad \rho_t=\rho(f_t),
\end{equation}
for all $\varphi\in C_b(\T^N)$. We are interested in equation \eqref{Eqrescaled}. The associated process is then $(X_{\eps^{-2}t} , V_{\eps^{-2}t})$ and the associated spatial density $\rho_{\eps^{-2}t}$. Our main result, Theorem~\ref{th:AD}, describes the limit behavior of $\rho_{\eps^{-2}t}$.

\subsection{Main result}

\paragraph{Notation.}
The three first moments of a function $f\in L^1(\R^N,|v|^2 dv)$ are written
\begin{equation}\label{3moments}
\rho(f)=\int_{\R^N}f(v)dv,\quad J(f)=\int_{\R^N} v f(v)dv,\quad K(f)=\int_{\R^N} v\otimes v f(v)dv,
\end{equation}
where $a\otimes b$ is the $N\times N$ rank-one matrix built on $a,b\in\R^N$ with $ij$-th elements $a_i b_j$. We use the notation
\begin{equation}\label{otimessym}
a\osym b=a\otimes b+b\otimes a
\end{equation}
to denote the symmetric version of $a\otimes b$. We denote by $K$ the second moment of $M$ (because $M$ is a Maxwellian, this is simply the identity matrix of size $N\times N$ here):
\begin{equation}\label{defK}
K=K(M)=\int_{\R^N} v\otimes v M(v) dv=\mathrm{Id}_N.
\end{equation}
For $m\in\N$, we denote by $\bar{J}_m(f)$ the total $m$-th moment of $f$:
\begin{equation}\label{barJm}
\bar{J}_m(f)=\iint_{\T^N\times\R^N}|v|^m f(x,v)dx dv.
\end{equation} 
Let us also introduce the Banach space
\begin{equation}\label{defG}
G_m=\left\{ f\in L^1(\T^N\times\R^N); \bar{J}_0(f)+\bar{J}_m(f)<+\infty\right\},
\end{equation}
with norm $\|f\|_{G_m}=\bar{J}_0(f)+\bar{J}_m(f)$. Eventually, we define the diffusion matrix $K_\sharp$ and the vector field $\Theta$ of the limit equation \eqref{eq:rho} by the formula
\begin{equation}
K_\sharp=K+\frac12\E\left[\bar{E}(0)\osym[R_0(\bar{E}(0))+(b-1)R_1(\bar{E}(0))]\right],\label{Ksharp}
\end{equation}
and
\begin{equation}
\Theta=\frac{b}{2}\div_x\E\left[R_1(\bar{E}(0))\osym\bar{E}(0)\right] +\E\left[R_1R_0(\bar{E}(0))\div_x(\bar{E}(0))\right],\label{PSI}
\end{equation}
where $b^\LB=2$ in the case $Q=\QLB$ and $b^\FP=1$ in the case $Q=\QFP$, and where the resolvent $R_\lambda$ is defined by the formula
\begin{equation}\label{def:resolvent}
R_\lambda\varphi(\e):=\int_0^\infty e^{-\lambda t} \mathtt{P}_t\varphi(\e)dt
\end{equation} 
In \eqref{def:resolvent}, $(\mathtt{P}_t)$ denotes the Markov semi-group generated by $(\bar{E}_t)$. Sufficient conditions for \eqref{def:resolvent} to make sense are given at the end of Section~\ref{sec:MixE}. For $i,j\in\{1,\ldots,N\}$, $x,y\in\T^N$, we set 
\begin{align}
H(i,x,j,y)=\frac12\E\left(\left[R_0(\bar{E}_0(x))\right]_i\left[\bar{E}_0(y)\right]_j+\left[R_0(\bar{E}_0(y))\right]_j\left[\bar{E}_0(x)\right]_i\right).\label{defKernelQ}
\end{align}
The function $H$ is a kernel on the space $L^2(\T^N;\R^N)$. The associated operator is denoted by $S$:
\begin{align}\label{OpS}
S\rho_i(x)=\sum_{j=1}^N\int_{\T^N} H(i,x,j,y)\rho_j(y)dy.
\end{align}
We show in Proposition~\ref{prop:Qtraceclass} that $S$ is symmetric, non-negative and trace-class. Our main result of diffusion-approximation for $\rho^\eps$ is the following one.

\begin{theorem} Let $f^\eps_\mathrm{in}\in G_3$ be non-negative. Let $m>N+1$.
Let $(\bar{E}_t)$ be a mixing force field on $H^{m}(\T^N;\R^N)$ according to Definition~\ref{def:mixE}. Let $f^\eps\in C([0,T];L^1(\T^N\times\R^N))$ be the mild solution to \eqref{Eqrescaled} with initial condition $f^\eps_\mathrm{in}$, in the sense of Definition~\ref{def:mildLB} or \ref{def:mildFP}, depending on the nature of the collision operator $Q$. Let $\rho^\eps=\rho(f^\eps)$. Assume the convergence 
\begin{equation}\label{cvinitial}
\rho(f^\eps_\mathrm{in})\to\rho_\mathrm{in}\mbox{ in }L^2(\T^N).
\end{equation} 
Let $K_\sharp$ and $\Theta$ be defined by \eqref{Ksharp} and \eqref{PSI} respectively. Then
$(\rho^\eps)$ converges in law on $C([0,T];H^{-1}(\T^N))$ to $\rho$, the weak-$L^1$ martingale solution in the sense of Definition~\ref{def:sollimeq} of the stochastic equation
\begin{equation}\label{eq:rho}
d\rho=\div_x(K_\sharp\nabla_x\rho+\Theta\rho)dt+\sqrt{2}\div_x(\rho S^{1/2}dW(t)),
\end{equation}
with initial condition
\begin{equation}\label{eq:rhoIC}
\rho(0)=\rho_\mathrm{in}.
\end{equation}
In \eqref{eq:rho}, $W(t)$ is a cylindrical Wiener process on $L^2(\T^N;\R^N)$, and $S$ is defined by \eqref{OpS}.
\label{th:AD}\end{theorem} 

\begin{remark}[Enhanced diffusion] The Stratonovitch formulation of \eqref{eq:rho} is 
\begin{equation}\label{eq:rhoStrato}
d\rho=\div_x(\tilde{K}_\sharp\nabla_x\rho+\tilde{\Theta}\rho)dt+\sqrt{2}\div_x(\rho\circ S^{1/2}dW(t)),
\end{equation}
where
\begin{equation}\label{tildeKsharp}
\tilde{K}_\sharp=K+\frac12(b-1)\E\left[R_1(\bar{E}(0))\osym\bar{E}(0)\right],
\end{equation}
with $b^\LB=2$, $b^\FP=1$. Lemma~\ref{lem:sympos} below shows that $K_\sharp\geq K$ and $\tilde{K}_\sharp\geq K$. Similar effects of enhanced diffusion in homogenization procedures are observed in \cite[Theorem~3.2]{Evans89} for example. In the Fokker-Planck case however, no additional diffusion appears when one uses the Stratonovitch form of the limit equation, since $\tilde{K}_\sharp=K$. Let us focus on the Linear Boltzmann case, or on the It\^o form of the limit equation. This last point of view is relevant if we focus on the average $r:=\E[\rho]$, that will be a solution to the equation
\begin{equation}\label{diffr}
\partial_t r-\div_x(K_\sharp\nabla_x r)=0.
\end{equation}
We examine first under what condition the matrix $\tilde{K}_\sharp$ may degenerate in the matrix $K$. When $b>1$, this happens only in the trivial case $\bar{E}_t\equiv 0$, as explained in Remark~\ref{rk:enhanced2}. Let us examine the matrix $K_\sharp$. When $b=1$, it may coincide with $K$ if $\E\left[R_0(\bar{E}(0))\osym \bar{E}(0)\right]=0$, but this happens only if, for all $q\in\R^N$, the map $\e\mapsto\e\cdot q$ is in the kernel of the Dirichlet form associated to $(\bar{E}_t)$ (the details are given in Remark~\ref{rk:enhanced2} also). It is easy to check that this condition will not be satisfied in many instances, like diffusion or jump processes. 
\label{rk:enhanced1}\end{remark}

\begin{remark}[Diffusion-approximation in the context of kinetic equations] The influence of stochastic mixing forcing terms in kinetic equations has also been investigated in \cite{PoupaudVasseur03,GoudonRousset09}. The context and the results in these two papers are different from the present one however. Indeed,
\begin{enumerate}
\item the starting kinetic equations in \cite{PoupaudVasseur03,GoudonRousset09} are not collisional,
\item In \cite{PoupaudVasseur03,GoudonRousset09}, in the scaling that is considered, a collisional kinetic equation is obtained at the limit. The collision operator (an operator acting on functions of the variable $v$ thus) is a diffusion operator. At the level of trajectories, the appearance of this operator is explained by the convergence of the velocity $V_t$ of particles to a diffusion like the one solving equation \eqref{trajFP} with $E=0$. 
\end{enumerate}
Let us also mention here the recent paper \cite{Goudon2020}, where the limit of the kinetic equation \eqref{Eqrescaled} is also investigated. The framework of \cite{Goudon2020} is deterministic, the oscillating forcing term $\bar{E}(t)$ being quasi-periodic. An homogenization procedure leads then to a drift-diffusion equation at the limit. Enhanced diffusion may be observed or not, depending on the nature of the collision operator, \cite[Lemma~3.6]{Goudon2020}.
\end{remark}

\begin{remark}[Weak convergence] Let us make some comments on the weak mode of convergence of $\rho^\eps$ in Theorem~\ref{th:AD}. It is weak in the probabilistic sense (convergence in law). This is inherent to the limit theorems (like the Donsker theorem) which lay the foundation of diffusion-approximation results. The convergence is weak with respect to the space-variable also. We obtain below a bound in $G_3$ on $f^\eps$ thus by interpolation a better convergence than convergence in $C([0,T];H^{-1}(\T^N))$ holds. But this is still in a space with negative regularity with respect to $x$. 
We intend to improve this point, and to consider non-linear equations in a similar regime, in a future work. Nevertheless, note that, in the very special case where $\bar E$ is independent of the space variable, strong convergence can be established. Indeed, the spatial derivatives
of $f^\eps$ then satisfy the same equation as $f^\eps$. Bounds in $L^1$ on the derivatives of $f^\eps$ can be obtained in this way, by using the estimate~\eqref{sol:fLB}.
\medskip

An other standard tool in the study of kinetic equations are entropy estimates. In our context, we are not able do establish such estimates. 
The lack of entropy estimates has several consequences. One of those is that we do not have any $L^2$-bound in space on $\rho^\eps$. We have some uniform bounds in $L^1$ however, and this is why we consider solutions to the limit SPDE \eqref{eq:rho} taking values in $L^1$ (see Definition~\ref{def:sollimeq}). For such weak solutions, proving uniqueness for the limit problem is problematic at first sight. We use a duality method, using a backward SPDE, to establish pathwise uniqueness: see Theorem~\ref{th:sollimeqEU}.
\end{remark}

The plan of the paper is the following. In Section~\ref{sec:MixE} we describe the type of forcing field $\bar{E}(t)$ which we consider. In Section~\ref{sec:unperturbed}, we prove some mixing properties and compute the invariant measures for the unperturbed equation~\eqref{Stov}. In Section~\ref{sec:Cyeps}, we solve the Cauchy Problem for the kinetic equation~\eqref{Perturb}.  In Section~\ref{sec:DA}, we establish our main result of diffusion-approximation, Theorem~\ref{th:AD}.

\section{Mixing force field}\label{sec:MixE}
Let  $m>N+1$ be a given integer and let $F=H^{2m}(\T^N;\R^N)$ be endowed with the norm
\begin{equation}\label{normF}
\|\e\|_F=\left[\sum_{|\alpha|\leq 2m}\sum_{i=1}^N\|\partial^\alpha\e_i\|^2_{L^2(\T^N)}\right]^{1/2}.
\end{equation}
where the first sum in \eqref{normF} is over all multi-indices $\alpha\in\N^N$ of length $|\alpha|=\alpha_1+\dotsb+\alpha_N$ less than $2m$. The space $F$ will be the state space for the mixing force field $\bar{E}$: we assume that we are given $(\bar{E}_t)_{t\geq 0}$, a stationary, homogeneous Markov process of generator $A$ over $F$ (the generator is defined according to the theory developed in Appendix~\ref{app2}). Let $P(t,\e,B)$ be a transition function for $(\bar{E}_t)$ associated to the filtration generated by $(\bar{E}_t)$ (see, e.g., \cite[p.~156]{EthierKurtz86} for the definition), satisfying the Chapman-Kolmogorov relation
\begin{equation}\label{ChapKol}
P(t+s,\e,B)=\int_F P(s,\e_1,B) dP(t,\e,d\e_1),
\end{equation}
for all $s,t\geq 0$, $\e\in F$, $B$ Borel subset of $F$. It will be helpful (and it is often more natural) to see $(\bar{E}_t)$ as the particular evolution $(E_t(\e))$ of a process starting from $\e$, when $\e$ is drawn according to the equilibrium measure. Let us give the details of this procedure: let $\mathcal{P}(F)$ be the set of Borel probability measures on $F$. By \cite[p.~157]{EthierKurtz86}, up to a modification of the probability space $(\Omega,\F)$, say into a probability space $(\tilde{\Omega},\tilde{\F})$, there exists a collection $\{\P_\mu;\mu\in\mathcal{P}(F)\}$ of probability measures and some Markov processes $(E(t,s))_{t\geq s}$ with transition function $P$ such that, $\P_\mu(E(s,s)\in D_0)=\mu(D_0)$ for all Borel subset $D_0$ of $F$. When $\mu$ is the Dirac mass $\mu=\delta_\e$, we use the shorter notation $\P_\e$ instead of $\P_{\delta_\e}$. By \cite[p.~157]{EthierKurtz86} additionally, for all $D\in\F$, $\e\mapsto\P_\e(D)$ is Borel measurable. Let $\e_0$ be a random variable on $F$ of law $\mu$. We do a slight abuse of notation and denote by $(E(t,s;\e_0),\P)$ the couple $(E(t,s),\P_\mu)$. This means that the finite-dimensional distribution of both processes are the same, \textit{i.e.} 
\begin{equation}\label{ppmu}
\P(E(t_1,s;\e_0)\in D_1,\ldots,E(t_n,s;\e_0)\in D_n)=\P_\mu(E(t_1,s)\in D_1,\ldots,E(t_n,s)\in D_n),
\end{equation}
for all $s\leq t_1\leq\cdots\leq t_n$, and $D_1,\ldots,D_n$ Borel subsets of $F$. For sim\-pli\-ci\-ty, we also use the notation $E(t;\e)$, or $E_t(\e)$, instead of $E(t,0;\e)$. Note that, by iteration of \eqref{ChapKol}, we have
\begin{multline}
\P(\bar{E}(0)\in D_0,\bar{E}(t_1)\in D_1,\ldots,\bar{E}(t_n)\in D_n)\\
=\int_{D_0}\cdots\int_{D_{n-1}}P(t_n-t_{n-1},\e_{n-1},D_n)P(t_{n-1}-t_{n-2},\e_{n-2},d\e_{n-1})\cdots P(t_1,\e_0,d\e_1)d\nu(\e_0)\\
=\P_\nu(E(t_1,0)\in D_1,\ldots,E(t_n,0)\in D_n),
\label{Pnu}\end{multline}
where $\nu$ is the law of $\bar{E}(0)$. Therefore $\bar{E}_t$ and $E_t(\bar{E}_0)$ have the same finite-dimensional distributions: $\bar{E}_t$ is a version $E_t(\bar{E}_0)$. The probability space $\tilde{\Omega}$ used in \cite[p.~157]{EthierKurtz86} to define the probability measures $\P_\e$ is the path-space $F^{[0,+\infty)}$ (the $\sigma$-algebra $\tilde{\F}$ is the product $\sigma$-algebra). Assume in addition that $(\bar{E}_t)$ is c{\`a}dl{\`a}g. Then it is clear that we can take the Skorohod space $D([0,+\infty);F)$ as a path space to define $\P_\e$. The $\sigma$-algebra $\tilde{\F}$ is then the trace of the product $\sigma$-algebra, which coincide with the Borel $\sigma$-algebra when the Skorokhod topology is considered on $D([0,+\infty);F)$. In this context, it holds true that $\e\mapsto\P_\e(D)$ is Borel measurable for
all $D\in\tilde{\F}$ (see the proof of Proposition~1.2 p.~158 in \cite{EthierKurtz86}). To sum up (see \cite[Section~I-3]{RevuzYor99}), if $(\bar{E}_t)$ is c{\`a}dl{\`a}g, we can assume that $t\mapsto E(t,s;\e)$ is c{\`a}dl{\`a}g, for all $s\in\R$ and $\e\in F$. As a last remark, note that it is always possible, using the Kolmogorov extension theorem, to build a c{\`a}dl{\`a}g stationary process $(\check{E}(t))_{t\in\R}$ indexed by $t\in\R$ with the finite-dimensional distributions
\begin{multline}\label{eternalstationary}
\P(\check{E}(s)\in D_0,\check{E}(s+t_1)\in D_1,\ldots,\check{E}(s+t_n)\in D_n)\\
=\P(\bar{E}(0)\in D_0,\bar{E}(t_1)\in D_1,\ldots,\bar{E}(t_n)\in D_n),
\end{multline}
for all $s\in\R$, $0\leq t_1,\ldots,t_n$. Instead of adding a new notation $(\check{E}(t))_{t\in\R}$, we simply denote this process by 
$(\bar{E}(t))_{t\in\R}$. We also denote by $(\G_t)$ the usual augmentation (see \cite[Definition~(4.13), Section~I-4]{RevuzYor99}) of the canonical filtration $(\F_t)$ on $D([0,+\infty);F)$ with respect to the family $(\P_\e)_{\e\in F}$. In successive order, $(\F_t)$ is the filtration generated by the evaluation maps $(\pi_t)$, $\pi_t(\omega)=\omega(t)$; $\F_t^*$ is the intersection over $\e\in F$ of the $\sigma$-algebras $\F_t^{\P_\e}$ obtained by completing $\F_t$ with $\P_\e$-negligible sets; and $\G_t$ is $\F^*_{t+}$:
\begin{equation}\label{defGt}
\G_t=\bigcap_{s>t}\F^*_s.
\end{equation}

\begin{definition}[Mixing force field] Let $(\bar{E}_t)_{t\geq 0}$ be a c{\`a}dl{\`a}g, stationary, homogeneous Markov process of generator $A$, in the sense of Appendix~\ref{app2}, over $F$. We say that $(\bar{E}_t)_{t\geq 0}$ is a mixing force field if the conditions \eqref{BallR}, \eqref{betacentred}, \eqref{mixCoupled0}, \eqref{normalizegammamix}, \eqref{AR0} below are satisfied.
\label{def:mixE}\end{definition}

Our first hypothesis is that there exists a stable ball: there exists $\mathtt{R}\geq 0$ such that: almost-surely, for all $\e$ with $\|\e\|_F\leq \mathtt{R}$, for all $t\geq 0$,
\begin{equation}\label{BallR}
\|E(t;\e)\|_F\leq \mathtt{R}.
\end{equation}

Our second hypothesis is about the law $\nu$ of $\bar{E}_t$. We assume that it is supported in the ball $\bar{B}_\mathtt{R}$ of $F$ (therefore, it has moments of all orders) and that it is centered:
\begin{equation}\label{betacentred}
\int_F \e\ d\nu(\e)=\E\left[\bar{E}_t\right]=0,
\end{equation}
for all $t\geq 0$. Note that a consequence of this hypothesis is that: almost-surely, for all $t\geq 0$,
\begin{equation}\label{RE}
\|\bar{E}_t\|_F\leq\mathtt{R}.
\end{equation}

Our third hypothesis is a mixing hypothesis: we assume that there exists a continuous, non-increasing, positive and integrable function $\gamma_\mathrm{mix}\in L^1(\R_+)$ such that, for all probability measures $\mu$, $\mu'$ on $F$, for all random variables $\e_0$, $\e_0'$ on $F$ of law $\mu$ and $\mu'$ respectively, there is a coupling $((E^*_t(\e_0))_{t\geq 0},(E^*_t(\e_0'))_{t\geq 0})$ of $((E_t(\e_0))_{t\geq 0},(E_t(\e_0'))_{t\geq 0})$ such that 
\begin{equation}\label{mixCoupled0}
\E\|E^*_t(\e_0)-E^*_t(\e_0')\|_F\leq\mathtt{R}\gamma_\mathrm{mix}(t),
\end{equation}
for all $t\geq 0$. Typically, we expect $\gamma_\mathrm{mix}$ to be of the form $\gamma_\mathrm{mix}(t)=C_\mathrm{mix}e^{-\beta_\mathrm{mix}t}$, $\beta_\mathrm{mix}>0$ (see the example treated in Section~\ref{sec:egmix} for instance). 

\subsection{Some consequences of the mixing hypothesis}\label{sec:clcoupled}

Let $\varphi$ be a Lipschitz continuous function on $F$. We have
$$
\E\varphi(E^*_t(\e_0))=\<\mathtt{P}_t\varphi,\mu\>
$$
(where $\mathtt{P}_t$ denote the semi-group associated to $A$: $\E_\e\varphi(E_t)=\mathtt{P}_t\varphi(\e)$). From \eqref{mixCoupled0}, it follows that
\begin{equation}\label{gammamixmumu}
\left|\<\mathtt{P}_t\varphi,\mu\>-\<\mathtt{P}_t\varphi,\mu'\>\right| \leq \|\varphi\|_{\mathrm{Lip}}\mathtt{R}\gamma_\mathrm{mix}(t),
\end{equation}
for all $t\geq 0$. Let $\nu$ denote the law of $(\bar{E}(t))$ and let $\e\in\bar{B}_\mathtt{R}$. We will use \eqref{gammamixmumu} in particular when $\e_0=\e$ a.s. and $\e_0'$ has law $\nu$. Then \eqref{gammamixmumu} gives the following mixing estimate:
\begin{equation}\label{gammamix}
\|\mathtt{P}_t\varphi(\e)-\<\varphi,\nu\>\|_F\leq\mathtt{R}\|\varphi\|_{\mathrm{Lip}}\gamma_\mathrm{mix}(t),
\end{equation}
for all $t\geq 0$, for all $\e\in\bar{B}_\mathtt{R}$. The estimate \eqref{gammamix} has an extension to quadratic functionals: for all linear and continuous $\Lambda\colon F\to \R$, for all bi-linear and continuous $q\colon F\times F\to\R$, we have, for all $\e\in\bar{B}_\mathtt{R}$,
\begin{equation}\label{gammamixquad}
\|\mathtt{P}_t[\Lambda+q](\e)-\<\Lambda+q,\nu\>\|_F\leq \mathtt{R}\left(\|\Lambda\|_{B(F)}+2\mathtt{R}\|q\|_{B(F\times F)}\right)\gamma_\mathrm{mix}(t),
\end{equation}
where $\|\Lambda\|_{B(F)}$ is the norm of the linear form of $\Lambda$ and $\|q\|_{B(F\times F)}$ is the norm of the bi-linear form of $q$.
Note that, actually, $\<\Lambda,\nu\>=0$ by \eqref{betacentred}. The factor $\mathtt{R}$ in front of $\|q\|_{B(F\times F)}$ in \eqref{gammamixquad} is due to the decomposition (recall that $\e_0=\e$ a.s. and $\e_0'$ has law $\nu$)
\begin{multline*}
\mathtt{P}_tq(\e)-\<q,\nu\>
=\E\left[q(E_t^*(\e_0),E_t^*(\e_0))-q(E_t^*(\e_0),\E^*_t(\e_0'))\right]\\
+\E\left[q(E_t^*(\e_0),E^*_t(\e_0'))-q(E^*_t(\e_0'),E^*_t(\e_0'))\right].
\end{multline*}
We have indeed
\begin{align*}
|\mathtt{P}_tq(\e)-\<q,\nu\>|&\leq\|q\|_{B(F\times F)}\E\left[(\|E_t^*(\e_0)\|_F+\|E^*_t(\e_0')\|_F)\|E_t^*(\e_0)-E^*_t(\e_0')\|_F\right]\\
&\leq 2\mathtt{R}\|q\|_{B(F\times F)}\E\|E_t^*(\e_0)-E^*_t(\e_0')\|_F\quad\mbox{by }\eqref{BallR},\\
&\leq 2\mathtt{R}^2\|q\|_{B(F\times F)}\gamma_\mathrm{mix}(t)\quad\mbox{by }\eqref{mixCoupled0}.
\end{align*} 
Without loss of generality (as we can rescale $\gamma_\mathrm{mix}$ if we rescale $\mathtt{R}$), we assume
\begin{equation}\label{normalizegammamix}
\|\gamma_\mathrm{mix}\|_{L^1(\R_+)}=1.
\end{equation}

Using \eqref{gammamix}, the resolvent \eqref{def:resolvent}
is well defined for all $\lambda\geq 0$, $\e\in\bar{B}_\mathtt{R}$ and all $\varphi\colon F\to\R$ which is Lipschitz continuous and satisfies the cancellation condition $\<\varphi,\nu\>=0$. Using \eqref{betacentred}, we can therefore define $R_\lambda\varphi_h(\e)$ for $\lambda\geq 0$, where $\varphi_h(\e)=\langle \e,h\rangle_{L^2(\T^N)}$. Moreover by \eqref{gammamix}, there exists  $T_\lambda\colon F\to F$ such that 
 $R_\lambda\varphi_h(\e)=\langle T_\lambda(\e),h\rangle_{L^2(\T^N)}$. 
By a slight abuse of notation, we write $R_\lambda(\e)=T_\lambda(\e)$. By \eqref{mixCoupled0} (with $\e_0=\e$ a.s. and $\e_0'\sim\nu$)  and \eqref{normalizegammamix}, we have 
\begin{equation}\label{RoE}
\|R_0(\e)\|_F\leq \mathtt{R},
\end{equation}
for all $\e$ with $\|\e\|_F\leq\mathtt{R}$. Eventually, let $\Lambda\colon F\to \R$ be a linear functional. Then, with the notations above, $\varphi_\Lambda:=\Lambda\circ R_0$ is a map $F\to\R$. The generator $A$ acts on $\varphi_\Lambda$ and on the square of $\varphi_\Lambda$ and we will assume that there exists a constant $C^0_\mathtt{R}\geq 0$ such that the following bounds are satisfied: 
\begin{equation}\label{AR0}
|[A|\varphi_\Lambda|^2](\e)|\leq C^0_\mathtt{R}\|\Lambda\|_{B(F)}^2,\quad |[A\varphi_\Lambda](\e)|\leq C^0_\mathtt{R}\|\Lambda\|_{B(F)},
\end{equation}
for all $\e$ with $\|\e\|_F\leq\mathtt{R}$. 

\begin{remark}\label{Rk:asbound} Hypothesis \eqref{BallR}, an almost sure bound, is quite strong. We use it in an essential way in the estimates obtained in Proposition~\ref{prop:MomentBound} (bounds on the moments in $v$ of the solution). It is possible to relax the hypothesis \eqref{BallR}, by considering for example that there exists some constants ${\RR}\geq 1$, $q>2$, such that
\begin{equation}\label{RelaxBallR}
\E\left[\|E(t;\e)\|_F^2\right]\leq {\RR}^2(1+\|\e\|_F^2),\quad \E\bigg[\sup_{t\in[0,1]}\|\bar{E}_t\|_F^q\bigg]\leq {\RR}^q,
\end{equation}
for all $t\geq 0$ and for all $\e\in F$. Such an extension requires a lot of work however, and we judged simpler to work under \eqref{BallR}. We give in Section~\eqref{sec:egmix} some examples of processes (jumps processes and diffusion processes) that are admissible in the sense of Definition~\ref{def:mixE}.  
\end{remark}

\subsection{Covariance}\label{sec:covariance}

Our mixing hypothesis has the following consequence on the covariances of $(E_t)$ and $(\bar{E}_t)$: let 
\begin{equation}\label{Gammae}
\Gamma_\e(s,t)=\E\left[E_s(\e)\otimes E_t(\e)\right],\quad \bar{\Gamma}(t)=\E\left[\bar{E}(t)\otimes \bar{E}(0)\right].
\end{equation}
Let $t\geq s\geq r\geq 0$. Conditioning on $\G_{t-s}$, we have
$$
\Gamma_\e(t-r,t-s)=\mathtt{P}_{t-s}(e^{(s-r)A}\theta\otimes\theta)(\e),\quad\theta(\e)=\e
$$
It follows from \eqref{gammamixquad} that, for all $\e$ with $\|\e\|_F\leq\mathtt{R}$,
\begin{equation}\label{MixGamma}
\|\Gamma_\e(t-r,t-s)-\bar{\Gamma}(s-r)\|_F\leq 2\mathtt{R}^2\gamma_\mathrm{mix}(t-s).
\end{equation}
Note also that $\bar{\Gamma}(-t)=\E\left[\bar{E}(0)\otimes \bar{E}(t)\right]$ so that, using the notation \eqref{otimessym}, we have
\begin{equation}\label{GammaGammacheck}
\bar{\Gamma}(t)+\bar{\Gamma}(-t)=\E\left[\bar{E}(t)\osym\bar{E}(0)\right].
\end{equation}

\subsection{Some simple examples}\label{sec:egmix}

Let $(E_n(\e))_{n\geq 0}$ be a Markov chain on $F$ with $E_0(\e)=\e$, and let $(N_t)_{t\geq 0}$ be a Poisson process of rate $1$ ($N_0=0$) independent of $(E_n)$. We assume that the ball $\bar{B}_\mathtt{R}$ of $F$ is stable by $(E_n)$, that $(E_n(\e))_{n\geq 0}$ has the invariant measure $\nu$ and the mixing property
\begin{equation}\label{mixn}
\E\|E_n^*(\e_0)-E_n^*(\e_0')\|\leq C\mathtt{R}\gamma^n,
\end{equation}
where $\gamma<1$ for a coupling $(E_n^*(\e_0),E_n^*(\e_0'))$ of $(E_n(\e_0),E_n(\e_0'))$. Let 
\begin{equation}\label{egE}
E(t,s;\e_0)=E_{N_{t-s}}(\e_0)
\end{equation}
and let $\bar{E}_t=E(t,0;\bar{\e}_0)$, where $\bar{\e}_0$ is a random variable of law $\nu$ independent of $(E_n)_{n\geq 0}$ and $(N_t)_{t\geq 0}$. Then $(\bar{E}_t)$ is a stationary process (it is a time-homogeneous Markov process and is initially at equilibrium). It is c{\`a}dl{\`a}g, it satisfies \eqref{BallR}, \eqref{betacentred} if $\nu$ is centered, and also \eqref{mixCoupled0} since
\begin{align*}
\E\|E_t^*(\e_0)-E_t^*(\e_0')\|_F
&=\sum_{n=0}^\infty\P(N_t=n)\E\|E_n^*(\e_0)-E_n^*(\e_0')\|_F\\
&\leq C\mathtt{R}\sum_{n=0}^\infty e^{-t}\frac{t^n}{n!}\gamma^n=C\mathtt{R}e^{-(1-\gamma)t}=:\mathtt{R}\gamma_\mathrm{mix}(t).
\end{align*}
Let us simplify still by considering the situation where $E_{n+1}(\e)$ is drawn \textit{independently} on $E_n(\e)$, with law $\nu$. We can then consider the synchronous coupling $(E_n^*(\e_0),E_n^*(\e_0'))$ of $(E_n(\e_0),E_n(\e_0'))$ which is such that 
$E_n^*(\e_0)=E_n^*(\e_0')$ for all $n\geq 1$. It gives us
$$
\E\|E_t^*(\e_0)-E_t^*(\e_0')\|_F\leq 2\mathtt{R}\P(N_t=0)=2\mathtt{R}e^{-t}.
$$
In addition, the semi-group, generator and resolvent $R_0$ have the explicit forms
$$
\mathtt{P}_t\varphi(\e)=e^{-t}\varphi(\e)+(1-e^{-t})\<\varphi,\nu\>,
$$
and
$$
A\varphi(\e)=\<\varphi,\nu\>-\varphi(\e),\quad R_0\varphi(\e)=\e.
$$
From these formula, we deduce the second inequality in \eqref{AR0} with $C^0_\mathtt{R}\geq\mathtt{R}$. The first inequality in \eqref{AR0} is obtained with any $C^0_\mathtt{R}\geq 2\mathtt{R}^2$. Note that we also have $R_1\varphi(\e)=\e$. The matrix $K_\sharp$ in \eqref{Ksharp} is therefore given by
\begin{equation}\label{Ksharpsharp}
K_\sharp=K+\theta\E\left[\bar{E}(0)\otimes\bar{E}(0)\right],\quad \theta:=1=\frac{b-1}{2}\in\{1/2,1\}.
\end{equation}
In particular, we have $K_\sharp>K$ as soon as $\E\left[|\bar{E}(0)\cdot\xi|^2\right]>0$ for a $\xi\in\R^N$.
\medskip

Diffusion processes can be used to give some other instances of admissible force field. We fix $N\in\N\cup\{+\infty\}$ and set
\begin{equation}\label{mtdiffusion}
\bar{E}_t=\sum_{j=1}^N a_j \bar{Y}^j_t\e_j,
\end{equation}
where $Y^1,Y^2,\dotsc$ are some i.i.d. processes with state space the interval $(-1,1)$ and $a_1,a_2,\dotsc$ some non-trivial real numbers converging fast enough to zero and $\{\e_j;1\leq j\leq N\}$ a family which is free in $F$.
For the process $Y$, we choose a diffusion process with state space $I=(-1,1)$: a diffusion process reflected or killed at the boundary of $I$, or a Sturm-Liouville Markov process with a drift that is singular at the boundary, \cite[Chapter~2]{BakryGentilLedoux14}.

\section{Unperturbed equation: ergodic properties}\label{sec:unperturbed}

We consider first the equation
\begin{equation}\label{vStov}
\partial_t f_t+\bar{E}(t)\cdot\nabla_v f_t=Qf_t\quad t>0,v\in\R^N, 
\end{equation}
where $Q=\QLB$ or $Q=\QFP$. In \eqref{vStov}, $\bar{E}(t)$ stands for $\bar{E}(x,t)$, where $(\bar{E}(t))$ is a mixing force field. We will not indicate the dependence with respect to $x$, which is a simple parameter here. \medskip

To find the invariant measure for \eqref{vStov}, we solve the equation starting from a given time $s\in\R$, and then let $s\to-\infty$. More precisely, given $\e\in\R^N$, we consider the following evolution equation:
\begin{equation}\label{vStovfroms}
\partial_t f_t+E(t,s;\e)\cdot\nabla_v f_t=Qf_t\quad t>s,v\in\R^N.
\end{equation}
Let $f\in L^1(\R^N)$ and $s\in\R$. The solution to \eqref{vStovfroms} with initial condition $f_{t=s}=f$ is
\begin{align}
f_{s,t}^{\LB}(v)=e^{-(t-s)}f&\left(v-\int_s^t E(r,s;\e) dr\right)\nonumber\\
&+\rho(f)\int_s^t e^{-(t-\sigma)}\left[M\left(v-\int_\sigma^t E(r,s;\e) dr\right)\right]d\sigma,\label{fLBEEstar}
\end{align}
when $Q=\QLB$, and
\begin{align}
f_{s,t}^{\FP}(v)
=e^{N(t-s)}\int_{\R^N} f\left(e^{(t-s)} v-\int_s^t e^{-(s-\sigma)} E(\sigma,s,\e)d\sigma +\sqrt{e^{2(t-s)}-1}\, w\right)M(w)dw,\label{fFPEEstar}
\end{align}
when $Q=\QFP$. A brief explanation to \eqref{fLBEEstar} and \eqref{fFPEEstar} is given in Appendix~\ref{app1}. By the term ``solution to \eqref{vStovfroms}", we mean weak solutions, \textit{i.e.} functions $f\in C([s,+\infty);L^1(\R^N))$ satisfying the identity 
\begin{equation*}
\<f_t,\varphi\>=\<f,\varphi\>+\int_s^t\<f_\sigma,E(\sigma,t;\e)\cdot\nabla_v\varphi\>+\<f_\sigma,Q^*\varphi\> d\sigma,
\end{equation*}
almost-surely, for all $\varphi\in C^\infty_c(\R^N)$, for all $t\geq s$. We may also consider mild solutions (this is equivalent, actually), as we do in Section~\ref{sec:Cyeps}. We do not need to be very specific on that point here. All that matters to us is to understand the limit behavior of $f_{s,t}$ defined by \eqref{fLBEEstar}-\eqref{fFPEEstar}  when $s\to-\infty$. This is the content of the following result. 

\begin{theorem}[Invariant solutions] Let $(\bar{E}(t))$ be a mixing force field in the sense of Definition~\ref{def:mixE}. Let $f_{s,t}^\LB$ and $f_{s,t}^\FP$ be defined by \eqref{fLBEEstar} and \eqref{fFPEEstar} respectively, with $\e\in\bar{B}_\mathtt{R}$. Then 
\begin{equation}\label{cvbarL1}
(f_{s,t}^\LB,E(t,s;\e))\to(\rho(f)\bar{M}_t^\LB,\bar{E}_t)\mbox{ and }(f_{s,t}^\FP,E(t,s;\e))\to(\rho(f)\bar{M}_t^\FP,\bar{E}_t)
\end{equation}
in law on $L^1(\R^N)\times\R^N$ when $s\to-\infty$, where $\bar{M}_t^\LB$ and $\bar{M}_t^\FP$ are defined by
\begin{align}
\bar{M}_t^\LB=\int_{-\infty}^t e^{-(t-\sigma)}\left[M\left(v-\int_\sigma^t \bar{E}(r) dr\right)\right]d\sigma,\label{MbarLB}
\end{align}
and
\begin{align}
\bar{M}_t^\FP=M\left(v-\int_{-\infty}^t e^{-(t-r)}\bar{E}(r) dr\right),\label{MbarFP}
\end{align}
respectively. 
\label{th:invsol}\end{theorem}

The equation \eqref{vStov} is conservative; the evolution takes place in some manifolds 
\[
\left\{f\in L^1(\R^N);\int_{\R^N}f dv=\rho\right\}
\]
indexed by a parameter $\rho\in\R$. For such a $\rho$, we denote by $\mu_\rho$ the invariant measure defined by
\begin{equation}
\label{def:murho}
\<\varphi,\mu_\rho\>=\E\varphi(\rho\bar{M}_t,\bar{E_t}),
\end{equation}
for all continuous and bounded function $\varphi$ on $L^1(\R^N)\times\R^N$. As a consequence of \eqref{cvbarL1}, the law of $(\rho\bar{M}_t,\bar{E_t})$ is independent on time, and is the law of the unique invariant measure for the dynamical system $(f_t,E_t)$ described by \eqref{vStov}.

\begin{remark} We will call $\bar{M}_t^\LB$ and $\bar{M}_t^\FP$ the ``invariant solutions'', since their laws are the invariant measure for \eqref{vStov}. Note that $(\bar{E}(r))$ in \eqref{MbarLB} and \eqref{MbarFP} is defined for all $r\in\R$ (see the discussion and convention of notations around \eqref{eternalstationary}). 
\end{remark}

\begin{remark} Let $\varphi$ be a bounded continuous function on $\R^N\times\R^N$. Similarly to \eqref{ffromlaw}, we have, by conditioning on the natural filtration $(\F^E_t)$ of $(E_t)$:
\begin{equation}\label{ffromlaw2}
\E\left[\varphi(V_{s,t},E(t,s;\e))\right]=\E\int_{\R^N} f_{s,t}(v)\varphi(v,E(t,s;\e)) dv,
\end{equation}
where $V_{s,t}$ is the solution to \eqref{trajLB} or \eqref{trajFP} (with $\bar{E}(t)$ instead of $\bar{E}(t,X_t)$) starting from $V_s$ at time $t=s$, where $V_s$ follows the law of density $f$ with respect to the Lebesgue measure on $\R^N$. Since
$$
\Phi\colon(f,\e)\mapsto\int_{\R^N}f(v)\varphi(v,\e)dv
$$
is continuous and bounded on $L^1(\R^N)\times\R^N$, we deduce from Theorem~\ref{th:invsol} that
\begin{equation}\label{ergodicVE}
\lim_{s\to-\infty}\E\left[\varphi(V_{s,t},E(t,s;\e))\right]=\<\mu_\rho,\varphi\>:=\rho\E\int_{\R^N} \bar{M}_t(v)\varphi(v,\bar{E}_t) dv,
\end{equation}
where $\rho=\rho(f)$.
\label{rk:ergodicVt}\end{remark}

The proof of Theorem~\ref{th:invsol} uses the estimates \eqref{MwL1} and \eqref{MwMzL1} in the following lemma.

\begin{lemma} For $w,z\in\R^N$, we have the estimates and identities 
\begin{align}
&\|M(\cdot-w)\|_{L^2(M^{-1})}^2=e^{|w|^2},\label{Mw}\\
&\|M(\cdot-w)-M(\cdot-z)\|_{L^2(M^{-1})}^2=e^{|w|^2}+e^{|z|^2}-2e^{w\cdot z},\label{MwMz}
\end{align}
in $L^2(M^{-1})$, and 
\begin{align}
&\|M(\cdot-w)\|_{L^1(\R^N)}=1,\label{MwL1}\\
&\|M(\cdot-w)-M(\cdot-z)\|_{L^1(\R^N)}\leq 2\wedge \left[\frac{|w-z|}{(1-|w-z|)^+}\right]^{1/2}\label{MwMzL1}
\end{align}
in $L^1(\R^N)$.
\label{lem:MwMz}\end{lemma}

\begin{proof}[Proof of Lemma~\ref{lem:MwMz}] Standard manipulations and identities for Gaussian densities give \eqref{Mw}, \eqref{MwMz} and \eqref{MwL1} (one can also use \eqref{devHermite} below to prove \eqref{Mw} and \eqref{MwMz}). By \eqref{MwL1} and the triangular inequality, we have the bound by $2$ in \eqref{MwMzL1}. To obtain the second estimate, we use the identity
$$
\|M(\cdot-w)-M(\cdot-z)\|_{L^1(\R^N)}=\|M(\cdot-w+z)-M\|_{L^1(\R^N)},
$$
and the expansion
\begin{equation}\label{devHermite}
M(v-w)=\frac{1}{(2\pi)^{N/2}}e^{-\frac{|v-w|^2}{2}}=M(w)\sum_{n\in\N^N}H_n(v)w^n,
\end{equation}
where $H_n$ is the $n$-th Hermite polynomial (see \cite[Section~1.1.1]{Nualart2006}). This yields the inequality
\begin{align*}
\|M(\cdot-w)-M\|_{L^1(\R^N)}\leq M(w)\sum_{n\in\N^N\setminus\{0\}}\|H_n\|_{L^1(\R^N)}|w|^n.
\end{align*}
Since $\|H_n\|_{L^1(\R^N)}\leq \|H_n\|_{L^2(M^{-1})}=\frac{1}{\sqrt{n!}}$ (see \cite[Lemma~1.1.1]{Nualart2006}), the Cauchy-Schwarz inequality yields, for $|w|< 1$,
$$
\|M(\cdot-w)-M\|_{L^1(\R^N)}\leq M(w)\left[\frac{e^{|w|}|w|}{1-|w|}\right]^{1/2}\leq \left[\frac{|w|}{1-|w|}\right]^{1/2}.
$$
Indeed, setting $a=|w|$, we have $a\in[0,1]$ and
$$
M(w)e^{|w|/2}=\left[\frac{1}{(2\pi)^{N}}e^{a-a^2}\right]^{1/2}\leq \left[\frac{1}{(2\pi)^{N}}e^{1/4}\right]^{1/2}\leq 1
$$
since $e^{1/4}\leq 2\pi$.
\end{proof}

\begin{proof}[Proof of Theorem~\ref{th:invsol}] Let $\e\in\bar{B}_\mathtt{R}$ $t\in\R$, let $\Phi\colon L^1(\R^N)\times F\to\R$ be a bounded and uniformly con\-ti\-nu\-ous function and let $\eps>0$. Our aim is to show that
\begin{equation}\label{CVinvsol0}
|\E\Phi(f_{s,t}(v),E(t,s;\e))-\E\Phi(\rho\bar{M}_t,\bar{E}_t)|<K\eps,
\end{equation}
for $s<\min(0,t)$, $|s|$ large enough, where $K$ is a finite constant (it will turn out that $K=5$, but this does not matter). Note that it is sufficient to consider uniformly continuous functions in \eqref{CVinvsol0}, see Proposition~I-2.4 in \cite{IkedaWatanabe1989}. We denote by $\eta$ a modulus of uniform continuity of $\Phi$ associated to $\eps$.\medskip

\textbf{Step 1. Reduction to the case $f\in L^2(M^{-1})$.} The maps $f\mapsto f_{s,t}$, $f\mapsto \rho(f)\bar{M}_t$ are continuous on $L^1$, uniformly in $s\leq t$:
$$
\|f_{s,t}\|_{L^1(\R^N)},\|\rho(f)\bar{M}_t^\LB\|_{L^1(\R^N)}\leq\|f\|_{L^1(\R^N)}.
$$
Using the uniform continuity of $\Phi$ on $K$, we have
$$
|\E\Phi(f_{s,t}(v),E(t,s;\e))-\E\Phi(\rho\bar{M}_t,\bar{E}_t)|<2\eps+|\E\Phi((\tilde{f})_{s,t},E(t,s;\e))-\E\Phi(\rho(\tilde{f})\bar{M}_t,\bar{E}_t)|
$$
if $\|f-\tilde{f}\|_{L^1(\R^N)}<\eta$. Therefore, to prove \eqref{CVinvsol0}, we turn to the case $f\in L^2(M^{-1})$.\medskip

\textbf{Step 2. Cut-off after time $s$.} For $s\leq t$, introduce
\begin{align}
\bar{M}_{s,t}^\LB=\int_{s}^t e^{-(t-\sigma)}\left[M\left(v-\int_\sigma^t \bar{E}(r) dr\right)\right]d\sigma,\label{MbarLBs}
\end{align}
and
\begin{align}
\bar{M}_{s,t}^\FP=M\left(v-\int_{s}^t e^{-(t-r)}\bar{E}(r) dr\right).\label{MbarFPs}
\end{align}
We have $\|\bar{M}_{s,t}^\LB-\bar{M}_{t}^\LB\|_{L^1(\R^N)}\leq e^{-(t-s)}$ by a direct computation and
$$
\|\bar{M}_{s,t}^\FP-\bar{M}_{t}^\FP\|_{L^1(\R^N)}\leq b\left(\int_{-\infty}^s e^{-(t-r)}\bar{E}(r) dr\right)
$$
where $b(|w-z|)$ is the right-hand side of \eqref{MwMzL1}. We use the bound 
\begin{equation}\label{bsubsqrt}
b(r)\leq \frac{\sqrt{5}}{2}r^{1/2}
\end{equation}
and \eqref{BallR} to obtain, almost-surely, $\|\bar{M}_{s,t}^\FP-\bar{M}_{t}^\FP\|_{L^1(\R^N)}\leq \frac{\sqrt{5}}{2}\mathtt{R}^{1/2}e^{-\frac12(t-s)}$. To sum up, in both the $\LB$ and $\FP$ case, we have a bound almost-sure on $\|\bar{M}_{s,t}-\bar{M}_{t}\|_{L^1(\R^N)}$ by a deterministic quantity which tends to $0$ when $t-s\to+\infty$. It follows that, for $t-s$ large enough, 
$$
|\E\Phi(\rho(f)\bar{M}_t,\bar{E}_t)-\E\Phi(\rho(f)\bar{M}_{s,t},\bar{E}_t)|<\eps.
$$
In the next step we prove that
\begin{equation}\label{CVinvsol1}
|\E\Phi(f_{s,t},E(t,s;\e))-\E\Phi(\rho(f)\bar{M}_{s,t})|<2\eps,
\end{equation}
for $t-s$ large enough. \medskip

\textbf{Step 3. Convergence in law.} Let $\e\in\bar{B}_\mathtt{R}$. Let $\e_0=\e$ a.s. and $\e_0'=\bar{E}_s$. Since $E(s,t;\e)$ has the same law as $E_{t-s}(\e_0)$ and $\bar{E}(t)$ has the same law as $E_{t-s}(\e_0')$, \eqref{mixCoupled0} gives a coupling
$$
(E(s,t;\e),\bar{E}(t))_{t\geq s}\to (E^*(s,t;\e),\bar{E}^*_t)_{t\geq s}
$$
such that
\begin{equation}\label{mixCoupleds}
\E\|E^*(t,s;\e)-\bar{E}^*_t\|_F\leq\mathtt{R}\gamma_\mathrm{mix}(t-s),
\end{equation}
for all $t\geq s$. We have
\begin{equation}\label{CVtilde}
\E\Phi(f_{s,t},E(t,s;\e))-\E\Phi(\rho(f)\bar{M}_{s,t},\bar{E}_t)=\E\Phi(f^*_{s,t},E^*(s,t;\e))-\E\Phi(\rho(f)\bar{M}^*_{s,t},\bar{E}^*_t),
\end{equation}
where the superscript star in $f_{s,t}$ and $\bar{M}_{s,t}$ indicates that $E(s,t;\e)$ has been replaced by $E^*(s,t;\e)$ and $\bar{E}(t)$ by $\bar{E}^*_t$. Since 
\begin{multline*}
|\E\Phi(f^*_{s,t},E^*(s,t;\e))-\E\Phi(\rho(f)\bar{M}^*_{s,t},\bar{E}^*_t)|\\
\leq
\eps
+\|\Phi\|_{BC}\left[\P(\|f^*_{s,t}-\rho(f)\bar{M}^*_{s,t}\|_{L^1(\R^N)}>\eta)+\P(\|E^*(s,t;\e)-\bar{E}^*_t\|_F>\eta)\right],
\end{multline*}
it is sufficient to prove that $f^*_{s,t}-\rho(f)\bar{M}^*_{s,t}\to0$ and $E^*(s,t;\e)-\bar{E}^*_t\to 0$ in probability on $L^1(\R^N)$ and $F$ respectively. We show the strongest (strongest, as is proved classically by means of the Markov inequality) property
\begin{equation}\label{CVtildeE}
\lim_{s\to-\infty}\E\|f^*_{s,t}-\rho(f)\bar{M}^*_{s,t}\|_{L^1(\R^N)}=0,\quad \lim_{s\to-\infty}\E\|E^*(s,t;\e)-\bar{E}^*_t\|_F=0.
\end{equation}
The second limit in \eqref{CVtildeE} is a consequence of \eqref{mixCoupleds}. Let us prove the first limit. Consider first the $\LB$ case. Using \eqref{MwL1} and the estimate $|\rho(f)|\leq\|f\|_{L^1(\R^N)}$, we have 
\begin{multline*}
\E\|f^{\LB,*}_{s,t}-\rho(f)\bar{M}^{\LB,*}_{s,t}\|_{L^1(\R^N)}\leq \|f\|_{L^1(\R^N)}e^{-(t-s)}\\
+\|f\|_{L^1(\R^N)}\E\int_s^t e^{-(t-\sigma)}b\left(\int_\sigma^t |E^*(r,s,\e)-\bar{E}^*(r)|dr\right) d\sigma,
\end{multline*}
where, as in \eqref{bsubsqrt}, we denote by $b(|w-z|)$ the right-hand side of \eqref{MwMzL1}. From \eqref{bsubsqrt} follows
$$
2b(r)\leq\eps+\frac{5}{4\eps}r.
$$
We deduce the estimate
\begin{multline*}
\E\|f^{\LB,*}_{s,t}-\rho(f)\bar{M}^{\LB,*}_{s,t}\|_{L^1(\R^N)}\leq \|f\|_{L^1(\R^N)}(e^{-(t-s)}+\eps)\\
+\frac{5}{4\eps}\|f\|_{L^1(\R^N)}\int_s^t e^{-(t-r)}\E|E^*(r,s,\e)-\bar{E}^*(r)| dr.
\end{multline*}
By \eqref{mixCoupleds}, this yields the following estimate:
\begin{align}
\E\|f^{\LB,*}_{s,t}-\rho(f)\bar{M}^{\LB,*}_{s,t}\|_{L^1(\R^N)}&\leq \|f\|_{L^1(\R^N)}\left(e^{-(t-s)}+\eps
+\frac{5\mathtt{R}}{4\eps}\int_s^t e^{-(t-r)}\gamma_\mathrm	{mix}(t-r)dr\right)\nonumber\\
&=\|f\|_{L^1(\R^N)}\left(e^{-(t-s)}+\eps+\frac{5\mathtt{R} }{4\eps}\int_0^{t-s} e^{r-(t-s)}\gamma_\mathrm	{mix}(r)dr\right).\label{CVtilde2}
\end{align}
We fix $r_1$ such that $\frac54 \mathtt{R}\int_{r_1}^\infty\gamma_\mathrm	{mix}(r)dr<\eps^2$. Then
$$
\frac54 \mathtt{R}\int_0^{t-s} e^{r-(t-s)}\gamma_\mathrm	{mix}(r)dr\leq\eps^2+\frac54 \mathtt{R}\int_0^{r_1}\gamma_\mathrm	{mix}(r)dr\, e^{r_1-(t-s)}<2\eps^2
$$
for $t-s$ large enough and \eqref{CVtildeE} follows from \eqref{CVtilde2}. In the $\FP$ case, we start first from the exponential estimate
\begin{equation}\label{cvOUE0}
\|f_{s,t}^{\FP}|_{E\equiv 0}-\rho(f)M\|_{L^2(M^{-1})}\leq e^{s-t}\|f\|_{L^2(M^{-1})}.
\end{equation}
In \eqref{cvOUE0}, $f_{s,t}^{\FP}|_{E\equiv 0}$ denotes the function \eqref{fFPEEstar} obtained when $E\equiv 0$. The estimate \eqref{cvOUE0} is a consequence of the dual estimate in $L^2(M)$ for functions $h$ such that $\<h,M\>_{L^2(\R^N)}=0$, see \cite[p.~179]{BakryGentilLedoux14}. It implies 
\begin{equation}\label{cvOUE0L1}
\|f_{s,t}^{\FP}|_{E\equiv 0}-\rho(f)M\|_{L^1(\R^N)}\leq e^{s-t}\|f\|_{L^2(M^{-1})}.
\end{equation}
The translations
$$
v\mapsto v-\int_s^t e^{-(t-\sigma)} \tilde{E}(\sigma,s,\e)d\sigma,\quad v\mapsto v-\int_{s}^t e^{-(t-\sigma)} \tilde{E}_s^*(\sigma)d\sigma,
$$ 
leave invariant the $L^1$-norm. Therefore \eqref{cvOUE0L1} yields
\begin{multline*}
\E\|f^{\FP,*}_{s,t}-\rho(f) \bar{M}^{\FP,*}_t\|_{L^1(\R^N)}\leq e^{s-t}\|f\|_{L^2(M^{-1})}\\
+|\rho(f)|\E\left\|M\left(\cdot-\int_s^t e^{-(t-\sigma)} \bar{E}^*(\sigma)d\sigma\right)-M\left(\cdot-\int_s^t e^{-(t-\sigma)} E^*(\sigma,s,\e)d\sigma\right)\right\|_{L^1(\R^N)}.
\end{multline*}
We conclude as in the case $Q=\QLB$ by means of \eqref{MwMzL1}.
\end{proof}
\section{Resolution of the kinetic equation}\label{sec:Cyeps}

We consider the resolution of the Cauchy problem of \eqref{Perturb} or \eqref{Eqrescaled} at fixed $\eps>0$. We set $\eps=1$ for simplicity. Then \eqref{Perturb} and \eqref{Eqrescaled} are the same equation
\begin{equation}\label{eps11}
\partial_t f+v\cdot\nabla_x f+\bar E(t,x)\cdot\nabla_v f=Qf.
\end{equation}
What is relevant actually is the dynamics given by $(f,\e)\mapsto (f_t,E_t(\e))$, where $f_t$ is the solution to the equation
\begin{equation}\label{eps1}
\partial_t f+v\cdot\nabla_x f+E(t,x)\cdot\nabla_v f=Qf,
\end{equation}
with $E(t,x)=E_t(\e(x))$. Therefore, this is \eqref{eps1} which we solve. We simply assume that $t\mapsto E(t,\cdot)$ is a c{\`a}dl{\`a}g function with values in $F$ (see Section~\ref{sec:MixE} for the definition of the state space $F$). In the particular case $E(t,x)=E_t(\e(x))$, we define in this way pathwise solutions. We solve the Cauchy Problem for \eqref{eps1} in the $\LB$-case and in the $\FP$-case in Section~\ref{sec:CauchyLB} and Section~\ref{sec:CauchyFP} respectively. Then, in Section~\ref{sec:MarkovLBFP}, we establish the Markov property of the process $(f_t,E_t(\e))$, where the first component $f_t$ is the solution to \eqref{eps1} with the forcing $E(t,x)=E_t(\e(x))$. 

\subsection{Cauchy Problem in the LB case}\label{sec:CauchyLB}

Let $t\mapsto E(t,\cdot)$ be a c{\`a}dl{\`a}g function with values in $F$. Let $\Phi_t(x,v)=(X_t(x,v),V_t(x,v))$ denote the flow associated to the field $(v,E(t,x))$:
\begin{align*}
\dot X_t=&V_t,\quad X_0=x,\\
\dot V_t=&E(t,X_t),\quad V_0=v.
\end{align*}
The partial map $(x,v)\mapsto\Phi_t(x,v)$ is a $C^1$-diffeomorphism of $\T^N\times\R^N$. We denote by $\Phi^t$ the inverse application: $\Phi^t\circ\Phi_t=\mathrm{Id}$. Note that $\Phi^t$ and $\Phi_t$ preserve the Lebesgue measure on $\T^N\times\R^N$. 

\begin{definition}[Mild solution, LB case] Let $f_\mathrm{in}\in L^1(\T^N\times\R^N)$. Assume $Q=\QLB$. A continuous function from $[0,T]$ to $L^1(\T^N\times\R^N)$ is said to be a mild solution to \eqref{eps1} with initial datum $f_\mathrm{in}$ if
\begin{equation}\label{mildLB}
f(t)=e^{-t}f_\mathrm{in}\circ\Phi^t+\int_0^t e^{-(t-s)}[\rho(f(s))\M]\circ\Phi^{t-s}ds,
\end{equation}
for all $t\in[0,T]$.
\label{def:mildLB}\end{definition}

\begin{proposition}[The Cauchy Problem, LB case] Let $f_\mathrm{in}\in L^1(\T^N\times\R^N)$. There exists a unique mild solution to \eqref{eps1} in $C([0,T];L^1(\T^N\times\R^N))$ with initial datum $f_\mathrm{in}$. It satisfies 
\begin{equation}\label{sol:fLB}
\|f(t)\|_{L^1(\T^N\times\R^N)}\leq\|f_\mathrm{in}\|_{L^1(\T^N\times\R^N)}\quad\mbox{for all }t\in [0,T].
\end{equation} 
If $f_\mathrm{in}\geq 0$, then $f(t)\geq 0$ for all $t\in[0,T]$ and \eqref{sol:fLB} is an identity. In addition, if $f_\mathrm{in}\in W^{k,1}(\T^N\times\R^N)$ with $k\leq 2$, then
\begin{equation}\label{LBaddreg}
\|f\|_{L^\infty(0,T;W^{k,1}(\T^N\times\R^N))}\leq C(k,T,f_\mathrm{in}),
\end{equation}
where the constant $C(k,T,f_\mathrm{in})$ depends on $k$, $T$, $N$, and on the norms 
$$
\sup_{t\in[0,T]}\|E(t,\cdot)\|_F\mbox{ and }\|f_\mathrm{in}\|_{W^{k,1}(\T^N\times\R^N)}
$$
only. Eventually, if $f_\mathrm{in}\in G_m$, then $f(t)\in G_m$ for all $t\in[0,T]$.
\label{prop:CYLB}\end{proposition}

\begin{proof}[Proof of Proposition~\ref{prop:CYLB}] Let $X_T$ denote the  space of continuous functions from $[0,T]$ to $L^1(\T^N\times\R^N)$. We use the norm
$$
\|f\|_{X_T}=\sup_{t\in[0,T]}\|f(t)\|_{L^1(\T^N\times\R^N)}
$$
on $X_T$. Note that
\begin{equation}\label{rhotof}
\|\rho(f)\|_{L^1(\T^N)}\leq\|f\|_{L^1(\T^N\times\R^N)}.
\end{equation}
Let $f\in X_T$. Assume that \eqref{mildLB} is satisfied. Then, by \eqref{rhotof}, we have
\begin{align*}
\|f(t)\|_{L^1(\T^N\times\R^N)}\leq & e^{-t}\|f_\mathrm{in}\|_{L^1(\T^N\times\R^N)}+\int_0^t e^{-(t-s)}\|f(s)\|_{L^1(\T^N\times\R^N)}ds.
\end{align*}
By Gronwall's Lemma applied to $t\mapsto e^t\|f(t)\|_{L^1(\T^N\times\R^N)}$, we obtain \eqref{sol:fLB} as an a priori estimate.
Besides, the $L^1$-norm of the integral term in \eqref{mildLB} can be estimated by $(1-e^{-T})\|f\|_{X_T}$. Therefore existence and uniqueness of a solution to \eqref{mildLB} in $L^1(\Omega;X_T)$ follow from the Banach fixed point Theorem. To obtain the additional regularity \eqref{LBaddreg}, we do the same kind of estimates on the system satisfied by the derivatives and incorporate these estimates in the fixed-point space. To conclude the proof, let us assume $f_\mathrm{in}\geq 0$. Since $s\mapsto s^-$ (negative part) is convex and satisfies $(a+b)^-\leq a^- +b^-$, we deduce from \eqref{mildLB} and the Jensen inequality that
$$
f^-(t)\leq \int_0^t e^{-(t-s)}[\rho(f(s))\M]^-\circ\Phi^{t-s}ds.
$$
Since $M\geq 0$ and $\rho(f)^-\leq\rho(f^-)$, \eqref{rhotof} yields the estimate 
$$
\|f^-(t)\|_{L^1(\T^N\times\R^N)}\leq \int_0^t e^{-(t-s)}\|f^-(s)\|_{L^1(\T^N\times\R^N)}ds.
$$ 
We conclude that $f^-=0$ by the Gr\"onwall Lemma. Eventually, that $f_\mathrm{in}\in G_m$ implies $f(t)\in G_m$ for all $t\in[0,T]$ (propagation of moments) is proved in Proposition~\ref{prop:MomentBound}.
\end{proof}

\subsection{Cauchy Problem in the FP case}\label{sec:CauchyFP}
 
Let $K_t(x,v;y,w)$ denote the kernel associated to the kinetic Fokker-Planck equation
\begin{equation}\label{kineticFP}
\partial_t f=\QFP f-v\cdot\nabla_x f.
\end{equation}
Let us recall some elementary facts about $K_t$ (see \cite{Bouchut93} for more results about the analytic properties of $K_t$, and \cite{PulvirentiSimeoni2000} for the probabilistic interpretation of $K_t$). The function $K_t(\cdot;y,w)$ is the density with respect to the Lebesgue measure on $\T^N\times\R^N$ of the law $\mu_t^{(y,w)}$ of the solution $(X_t,V_t)$ to the SDE
\begin{align}
dX_t=&V_tdt, &X_0=y,\label{stoFPX}\\
dV_t=&-V_tdt+\sqrt{2}dB_t, &V_0=w\label{stoFBV}.
\end{align}
where $B_t$ is a Wiener process over $\R^N$. Therefore 
$$
K_t f(x,v):=\iint_{\T^N\times\R^N}K_t(x,y;y,w)f(y,w)dydw
$$
satisfies the identity
\begin{equation}\label{Ktfphi}
\<K_tf,\varphi\>=\iint_{\T^N\times\R^N}\E\varphi(X_t,V_t)f(y,w)dydw,
\end{equation}
for $f\in L^1(\T^N\times\R^N)$ and $\varphi\colon\T^N\times\R^N\to\R$ continuous and bounded. The solution to \eqref{stoFPX}-\eqref{stoFBV} is given explicitly by
\begin{equation}\label{solXV}
\begin{split}
X_t=&y+(1-e^{-t})w+\int_0^t (1-e^{-(t-s)})dB_s,\\
V_t=&e^{-t}w+\int_0^t e^{-(t-s)}dB_s.
\end{split}
\end{equation}
The process $(X^0_t,V^0_t)$ given by \eqref{solXV} when $y=0$, $w=0$ is a Gaussian process with covariance matrix
\begin{equation}\label{GaussianQt}
Q_t:=\begin{pmatrix}
\displaystyle \int_0^t|1-e^{-s}|^2 ds & \int_0^t e^{-s}(1-e^{-s})ds \\
\displaystyle \int_0^t e^{-s}(1-e^{-s})ds & \int_0^t e^{-2s} ds
\end{pmatrix}\otimes\mathrm{I}_N.
\end{equation}
Using \eqref{GaussianQt} and \eqref{Ktfphi}-\eqref{solXV}, one can show that $K_t\colon L^p(\T^N\times\R^N)\to L^p(\T^N\times\R^N)$ with norm bounded by $e^{\frac{N}{p'}t}$. We have also the estimate
\begin{equation}\label{regKt}
\iint_{\T^N\times\R^N}|\nabla_w K_t(x,v;y,w)|dx dv\leq Ct^{-1/2},
\end{equation}
for all $(y,w)\in \T^N\times\R^N$, $t\in[0,T]$, with a constant $C$ independent of $(y,w)$ and $T$. The estimate \eqref{regKt} also follows from the estimate between Equations (26) and (27) of  \cite{Bouchut93}.
\begin{definition}[Mild solution, FP case] Let $t\mapsto E(t,\cdot)$ be a c{\`a}dl{\`a}g function with values in $F$. Let $p\in[1,+\infty)$. Let $f_\mathrm{in}\in L^p(\T^N\times\R^N)$. Assume $Q=\QFP$. A continuous function from $[0,T]$ to $L^p(\T^N\times\R^N))$ is said to be a mild solution to \eqref{eps1} in $L^p$ with initial datum $f_\mathrm{in}$ if
\begin{equation}\label{mildFP}
f(t)=K_tf_\mathrm{in}+\int_0^t \nabla_w K_{t-s}[E(s) f(s)]ds,
\end{equation}
for all $t\in[0,T]$.
\label{def:mildFP}\end{definition}

\begin{proposition}[The Cauchy Problem, FP case] Let $t\mapsto E(t,\cdot)$ be a c{\`a}dl{\`a}g function with values in $F$. Let $p\in[1,+\infty)$. Let $f_\mathrm{in}\in L^p(\T^N\times\R^N)$. Then \eqref{eps1} has a unique mild solution $f$ in $L^p$ with initial datum $f_\mathrm{in}$. If $f_\mathrm{in}\geq 0$, then $f(t)\geq 0$, for all $t\in[0,T]$. In addition, for every $k\leq 2$, the regularity $W^{k,p}(\T^N\times\R^N)$ is propagated: 
\begin{equation}\label{SobolevFP}
\sup_{t\in[0,T]}\|f(t)\|_{W^{k,p}(\T^N\times\R^N)}\leq C(k,T)\|f_\mathrm{in}\|_{W^{k,p}(\T^N\times\R^N)},
\end{equation}
where the constant $C(k,T)$ depends on $k$, $T$, $N$ and $\sup_{t\in[0,T]}\|E(t,\cdot)\|_F$. If $p=1$ and $f_\mathrm{in}\geq 0$, then $\|f(t)\|_{L^1(\T^N\times\R^N)}=\|f_\mathrm{in}\|_{L^1(\T^N\times\R^N)}$. If, more generally, there is no sign condition on $f_\mathrm{in}\in L^1(\T^N\times\R^N)$, then \eqref{sol:fLB} is satisfied. Eventually, if $f_\mathrm{in}\in G_m$, then $f(t)\in G_m$ for all $t\in[0,T]$.
\label{prop:CYFP}\end{proposition} 
 
\begin{proof}[Proof of Proposition~\ref{prop:CYFP}] The existence-uniqueness follows from the Banach fixed point Theorem using \eqref{regKt}, in a manner similar to the proof of Proposition~\ref{prop:CYLB}. To obtain \eqref{SobolevFP} for $k=1$, we assume first that $f(t)$ is in $W^{k,p}(\T^N\times\R^N)$ for all $t$ and we use the relations
\begin{align*}
\nabla_x K_t(x,v;y,w)&=-\nabla_y K_t(x,v;y,w),\\
\nabla_v K_t(x,v;y,w)&=-(1-e^{-t})\nabla_y K_t(x,v;y,w)-e^{-t}\nabla_w K_t(x,v;y,w),
\end{align*}
and Gronwall's Lemma, to obtain \eqref{SobolevFP}. We can drop the a priori requirement that $f(t)$ is in $W^{k,p}(\T^N\times\R^N)$ for all $t$ either by incorporating this in the fixed-point space, or by working with differential quotients. The case $k=2$ is obtained similarly. To prove that $f_\mathrm{in}\geq 0$ implies $f(t)\geq 0$, we use a duality argument: it is sufficient to prove the propagation of the sign for $L^\infty$ solutions to the dual equation
\begin{align}
\varphi(T)&=\psi,\label{dual1}\\
\partial_t\varphi&=-v\cdot\nabla_x\varphi-\bar{E}_t\cdot\nabla_v\varphi-\QFP^*\varphi,\quad 0<t<T.\label{dual2}
\end{align}
This follows from the maximum principle, since $\QFP^*\varphi=\Delta_v\varphi-v\cdot\nabla_v\varphi$. The maximum principle for the solutions to \eqref{dual1}-\eqref{dual2} also yields the $L^1$-estimate \eqref{sol:fLB}. The propagation of moments is proved in Proposition~\ref{prop:MomentBound}.
\end{proof}

\subsection{Markov property}\label{sec:MarkovLBFP}

We prove the following result.

\begin{theorem}[Markov property] Let $(\bar{E}(t))$ be a mixing force field in the sense of Definition~\ref{def:mixE}. We denote by $A$ the generator of $(\bar E_t)$. Let $\mathcal{X}$ denote the state space
\begin{equation}\label{defXX}
\mathcal{X}=L^1(\T^N\times\R^N)\times F.
\end{equation}
For $(f,\e)\in\mathcal{X}$, let $f_t$ denote the mild solution to \eqref{eps1} with initial datum $f$ and forcing $E_t(\e)$. Then $(f_t,E_t(\e))_{t\geq 0}$ is a time-homogeneous Markov process over $\mathcal{X}$. 
\label{th:MarkovPty}\end{theorem}

\begin{proof}[Proof of Theorem~\ref{th:MarkovPty}] We will just give the sketch of the proof. We use the propagation of the $W^{2,1}$-regularity stated in Proposition~\ref{prop:CYLB} and Proposition~\ref{prop:CYFP}. when $f_t$ has the regularity $W^{2,1}(\T^N\times\R^N)$, it is simple to prove that
\begin{equation}\label{Markov:Psi}
f_t=\Psi_{t}(f,(E(\sigma))_{0\leq\sigma\leq t}),
\end{equation}
where $\Psi_{0,t}(f,\cdot)$ is a \emph{continuous} map from $L^1([0,t];F)$ to $L^1(\T^N\times\R^N)$. Indeed, if $f^i_t$, $i\in\{1,2\}$ are two solutions to \eqref{eps1} corresponding to two different forcing terms $E^i(t,x)$, $i\in\{1,2\}$, we just need to write
$$
\left[\partial_t+E^1\cdot\nabla_v-Q\right](f^1_t-f^2_t)=(E^2-E^1)\cdot\nabla_v f^2_t,
$$
multiply the equation by $\sgn(f_1-f_2)$ and integrate, to obtain
\begin{equation}\label{f1f2}
\|f^1_t-f^2_t\|_{L^1(\T^N\times\R^N)}\leq C\int_0^t\|E^2(s)-E^1(s)\|_Fds,
\end{equation}
where the constant $C$ depends on the $L^\infty_t L^1_{x,v}$-norm of $\nabla_v f^2_t$. Without loss of generality, we can assume that $\Omega$ is the path-space, in which case \eqref{Markov:Psi} gives
\begin{equation}\label{Markov:Psiomega}
f_t=\Psi_{t}(f,\omega).
\end{equation}
Setting $\theta_t\omega=\omega(t+\cdot)$, we see that $\Psi$ satisfies the co-cycle property $\Psi_{t+s}(f,\omega)=\Psi_{t}(\Psi_s(f,\omega),\theta_s\omega)$. In this context of random dynamical system, it is clear that the process $(f_t,E_t(\e))_{t\geq 0}$ is a Markov process, \cite{Crauel91}. The extension to the case where $f\in L^1(\T^N\times\R^N)$ results from a density argument.
\end{proof}

Let us introduce the operators 
\begin{align}
\L_\sharp\varphi(f,\e)=&A\varphi(f,\e)+(Qf-\e\cdot\nabla_v f,D_f\varphi(f,\e)),\label{Lsharp}\\
\L_\flat\varphi(f,\e)=&-(v\cdot\nabla_x f,D_f\varphi(f,\e)),\label{Lflat}
\end{align}
and $\L=\L_\sharp+\L_\flat$. Formally, $\L$ is the generator associated to the Markov process $(f_t,E_t)$. In the proposition~\ref{prop:admtest} below, we describe a class of test-functions that
are in the domains of both $\L_\sharp$ and $\L_\flat$, the class being big enough to be used to characterize the limit process by the perturbed test-function method.

\begin{proposition} Let $(\bar{E}(t))$ be a mixing force field in the sense of Definition~\ref{def:mixE}. Let $A$ be the generator of $(E_t)$, let $\mathcal{X}$ be the state space defined by \eqref{defXX}, and let $\L_\sharp$ and $\L_\flat$ be defined by \eqref{Lsharp}-\eqref{Lflat}. Let $\psi\colon\R^m\times F\to\R$ be a continuous function which is bounded on bounded sets of $\R^m\times F$ and satisfies the following properties:
\begin{enumerate}
\item\label{item:1psi} for all $u\in\R^m$, $\e\mapsto\psi(u;\e)$ is in the domain of $A$ and $(u,\e)\mapsto A\psi(u;\e)$ is bounded on bounded sets of $\R^m\times F$,
\item\label{item:2psi} for all $\e\in F$, $u\mapsto\psi(u;\e)$ is differentiable, $(u,\e)\mapsto \nabla_u\psi(u;\e)$ is bounded on bounded sets of $\R^m\times F$ and continuous with respect to $\e$.
\end{enumerate}
Let $\xi_1,\ldots,\xi_m\in C^\infty_c(\T^N\times\R^N)$. Then the test-function
\begin{equation}\label{eq:admtest}
\varphi\colon (f,\e)\mapsto \psi(\<f,\xi_1\>,\ldots,\<f,\xi_m\>;\e)
\end{equation}
satisfies $\L_\sharp\varphi(f,\e),\L_\flat\varphi(f,\e)<+\infty$ for all $(f,\e)\in\mathcal{X}$ and $\varphi$ is in the domain of $\L$ in the sense that 
\begin{equation}\label{inDL}
P_t\varphi(f,\e)=\varphi(f,\e)+t\L\varphi(f,\e)+o_{f,\e}(t),
\end{equation} 
for all $(f,\e)\in\mathcal{X}$. 
\label{prop:admtest}\end{proposition}

\begin{proof}[Proof of Proposition~\ref{prop:admtest}] Let $\xi=(\xi_i)_{1,m}$. We have
\begin{align*}
\L_\sharp\varphi(f,\e)&=\big\{A\psi(u;\e)+\<f,Q^*\xi+\e\cdot\nabla_v\xi\>\nabla_u\psi(u;\e)\big\}\big|_{u=\<f,\xi\>},\\
\L_\flat\varphi(f,\e)&=\<f,v\cdot\nabla_x\xi\>\nabla_u\psi(u;\e)\big|_{u=\<f,\xi\>},
\end{align*}
therefore $(f,\e)\mapsto (\L_\sharp\varphi(f,\e),\L_\flat\varphi(f,\e))$ is bounded on bounded sets of $\mathcal{X}$. To obtain \eqref{inDL}, we use the decomposition of $P_t\varphi(f,\e)-\varphi(f,\e)$ into the sum of the terms
\begin{equation}\label{inDL1}
\E_{(f,\e)}\varphi(f,E_t)-\varphi(f,\e)
\end{equation}
and
\begin{equation}\label{inDL2}
\E_{(f,\e)}\left[\varphi(f_t,E_t)-\varphi(f,E_t)\right].
\end{equation}
By item~\ref{item:1psi}, we have the asymptotic expansion $\eqref{inDL1}=tA\psi(u;\e)\big|_{u=\<f,\xi\>}+o(t)$. In addition, by \eqref{eps1}, we have
$$
u_t=u+t\big(\<f,Q^*\xi+\e\cdot\nabla_v\xi\>+\<f,v\cdot\nabla_x\xi\>\big)+o(t),
$$
where $u_t=\<f_t,\xi\>$, $u=\<f,\xi\>$. By item~\ref{item:2psi}, we obtain the asymptotic expansion 
$$
\eqref{inDL2}=t\big(\<f,Q^*\xi+\e\cdot\nabla_v\xi\>+\<f,v\cdot\nabla_x\xi\>\big)\nabla_u\psi(u;\e)\big|_{u=\<f,\xi\>}+o(t).
$$
This concludes the proof.
\end{proof}

\begin{remark} The result of Proposition~\ref{prop:admtest} holds true if we consider some functions $\xi_i$ not as smooth and localized as $C^\infty_c$ functions, provided there is a sufficient balance with the regularity and integrability properties of $f$. For example, we apply Proposition~\ref{prop:admtest} in Section~\ref{sec:phi12} with $\xi_i(x,v)=\hat{\xi}_i(x)\zeta_i(v)$, where $\hat{\xi}_i$ is in some Sobolev space $H^s(\T^N)$ and $\zeta_i(v)$ is a polynomial in $v$ of degree less than two. In that case, we view $(f_t,E_t)$ as a Markov process on $\mathcal{X}_3:=G_3\times F$ and the conclusion of Proposition~\ref{prop:admtest} is valid for $f\in G_3$.
\label{rk:admtest}\end{remark}

\begin{remark} Note that, in the context of Proposition~\ref{prop:admtest}, the function $|\psi|^2$ has the same properties (item~\ref{item:1psi} and item~\ref{item:2psi}) as $\psi$. Therefore $|\psi|^2$ is also in the domain of $\L$.
\label{rk:admtest2}\end{remark}

\section{Diffusion-approximation}\label{sec:DA}

We consider the Markov process $(f^\eps_t,\bar{E}^\eps_t)$ (see Theorem~\ref{th:MarkovPty}). The generator $\L^\eps$ of this process can be decomposed as
$$
\L^\eps=\frac{1}{\eps^2}\L_\sharp+\frac{1}{\eps}\L_\flat,
$$
where $\L_\sharp$ and $\L_\flat$ are defined by \eqref{Lsharp} and \eqref{Lflat} respectively. For every $\varphi$ in the domain of $\L^\eps$, the process
\begin{equation}\label{martingaleeps}
M^\eps_\varphi(t):=\varphi(f^\eps_t,\bar{E}^\eps_t)-\varphi(f_\mathrm{in},\bar{E}_0)-\int_0^t\L^\eps\varphi(f^\eps_s,\bar{E}^\eps_s)ds
\end{equation}
is a $(\G_{t/\eps^2})$-martingale (this is a consequence of Theorem~\ref{th:MarkovPty} and Theorem~\ref{th:XMartingaleE} in Appendix~\ref{app2}). The equation associated to the principal generator $\L_\sharp$ is \eqref{Stov}. It has been analyzed in Section~\ref{sec:unperturbed}. Our approach to the proof of the convergence of $(\rho^\eps)$ uses the perturbed test-function method introduced by Papanicolaou, Stroock, Varadhan in \cite{PapanicolaouStroockVaradhan77}.
Let us explain the main steps of the proof.
\begin{enumerate}
\item \textbf{Limit generator.} To find the limit generator $\L$ associated to the equation satisfied by the limit $\rho$ of $(\rho^\eps)$, which acts on test functions $\varphi(\rho)$, we seek two correctors $\varphi_1$ and $\varphi_2$ such that, for the perturbed test function 
\begin{equation}\label{phiperturbed}
\varphi^\eps(f,\e)=\varphi(\rho)+\eps\varphi_1(f,\e)+\eps^2\varphi_2(f,\e),
\end{equation}
we may write $\L^\eps\varphi^\eps=\L\varphi+o(1)$. See Section~\ref{sec:perturbed}.
\item \textbf{Tightness.} We prove the tightness of the sequence $(\rho^\eps)$ in an adequate space. First, we obtain some bounds uniform with respect to $\eps$ by perturbation of the functional which we try to estimate. See Section~\ref{sec:bounds}. Then we establish some uniform estimates on the time increments of $(\rho^\eps)$. See Section~\ref{sec:tight}.
\item\textbf{Convergence.} We use the characterization of \eqref{eq:rho}-\eqref{eq:rhoIC} as a martingale problem to take the limit of the processes $(\rho_\eps)$. This is a very classical approach to the convergence of stochastic processes, see the introduction to \cite[Chapter III]{JacodShiryaev03}. We will consider the class $\varTheta$ of test-functions $\varphi(\rho)$ of the form 
\begin{equation}\label{sepclass}
\varphi(\rho)=\psi\left(\<\rho,\xi\>\right),
\end{equation}
for $\xi\in C^3(\T^N)$, $\rho\in L^1(\T^N)$, and $\psi$ a Lipschitz function on $\R$ such that $\psi'\in C^\infty_b(\R)$. This class $\varTheta$ is a separating class in $L^1(\T^N)$: if two random variables $\rho_1$ and $\rho_2$ satisfy $\E\varphi(\rho_1)=\E\varphi(\rho_2)$ for all $\varphi$ as in \eqref{sepclass}, then $\rho_1$ and $\rho_2$ have the same laws (this is because $\varTheta$ separates points, see Theorem~4.5 p.~113 in \cite{EthierKurtz86}). 
\end{enumerate}

\subsection{Perturbed test-function}\label{sec:perturbed}

Let $\varphi\colon L^1(\T^N)\to\R$ be a given test-function as in \eqref{sepclass}.
Consider the perturbation \eqref{phiperturbed}. To obtain the approximation $\L^\eps\varphi^\eps=\L\varphi+o(1)$, we identify the powers in $\eps$ in each side of this equality. This gives, for the scale $\eps^{-2}$, the first equation $\L_\sharp\varphi=0$. This equation is satisfied since $\varphi$ is independent of $\e$. Indeed, we have $A\varphi=0$, consequently, and also
$$
(Qf-\e\cdot\nabla_v f,D_f\varphi(\rho))=(\rho(Qf-\e\cdot\nabla_v f),D_\rho\varphi(\rho))=0
$$
since $\rho(Qf)=0$ and $\rho(\e\cdot\nabla_v f)=0$. At the scales $\eps^{-1}$ and $\eps^0$ respectively, we obtain the equation for the first corrector 
\begin{equation}\label{eq:phi1}
\L_\sharp\varphi_1+\L_\flat\varphi=0
\end{equation} 
and the equation for the second corrector
\begin{equation}\label{eq:phi2}
\L_\sharp\varphi_2+\L_\flat\varphi_1=\L\varphi.
\end{equation} 
If \eqref{eq:phi1} and \eqref{eq:phi2} are satisfied, then $\L^\eps\varphi^\eps=\L\varphi+\eps\L_\flat\varphi_2$. We solve 
\eqref{eq:phi1} and \eqref{eq:phi2} by formal computations first, see Section~\ref{sec:phi1} and Section~\ref{sec:phi2}. In Section~\ref{sec:phi12} then, we give and prove the rigorous statement concerning the resolution of  \eqref{eq:phi1} and \eqref{eq:phi2}, see Proposition~\ref{prop:firstsecondcorrector}.

\subsubsection{First corrector}\label{sec:phi1}
We seek a solution to \eqref{eq:phi1} by means of the resolvent formula
$$
\varphi_1(f,\e)=\int_0^\infty\E_{(f,\e)}\psi(f_t,E_t)dt,\quad\psi=\L_\flat\varphi,
$$
where $f_t$ is obtained either by \eqref{fLBEEstar} or \eqref{fFPEEstar} with $s=0$. The right-hand side $\psi$ is
\[
\psi(f,\e)=\L_\flat\varphi(f,\e)=-(\div_x(vf),D_f\varphi(\rho)).
\]
Since $\rho(vf)=J(f)$, this gives
\begin{equation}\label{psivarphi1}
\psi(f,\e)=-(\div_x(J(f)),D_\rho\varphi(\rho)).
\end{equation}
\begin{lemma} Let $f_{s,t}$ be equal either to \eqref{fLBEEstar}, or to \eqref{fFPEEstar}. The two first moments of $f_{s,t}$ (see \eqref{3moments} for the definition of the moments) are, respectively, $\rho(f_{s,t})=\rho(f)$, and
\begin{align}
J(f_{s,t})
=&e^{-(t-s)}J(f)
+\rho(f)\int_s^t e^{-(t-\sigma)} E(\sigma,s;\e) d\sigma.\label{Jfst}
\end{align}
\label{lem:Momentsft}\end{lemma}

\begin{proof}[Proof of Lemma~\ref{lem:Momentsft}] We use the formula
\begin{equation}\label{Moments}
\int_{\R^N}(1, v, v^{\otimes 2})M(v-w)dv=(1, w, K+w^{\otimes 2}),
\end{equation}
where $K$ is defined by \eqref{defK}. By \eqref{Moments} (and a change of variable in the $\FP$-case), we obtain \eqref{Jfst}.
\end{proof}

\begin{remark} Similar computations done on the equilibria $\bar{M}^\LB_t$ and $\bar{M}^\FP_t$ defined by \eqref{MbarLB} and \eqref{MbarFP} give the formula
\begin{equation}\label{Jstar}
J(\bar{M}^\LB_0)=J(\bar{M}^\FP_0)=\int_{-\infty}^0 e^{\sigma} \bar{E}(\sigma)d\sigma.
\end{equation}
\end{remark}

Using \eqref{gammamix} and \eqref{Jfst}, it is also simple to establish 
\begin{equation}\label{invJf}
\int_0^\infty |\E J(f_{0,t})| dt<+\infty,\quad\int_0^\infty \E J(f_{0,t}) dt=J(f)+\rho(f)R_0(\e).
\end{equation}

Combining \eqref{psivarphi1} and \eqref{invJf}, we obtain the following candidate as first corrector:
\begin{equation}\label{varphi1}
\varphi_1(f,\e)=-(\div_x(H(f,\e)),D_\rho\varphi(\rho)),\quad H(f,\e):=J(f)+\rho(f)R_0(\e).
\end{equation}

\subsubsection{Second corrector and limit generator}\label{sec:phi2}

Let $\mu_\rho$ be the invariant measure pa\-ra\-me\-tri\-zed by $\rho$ associated to $\L_\sharp$, defined by \eqref{def:murho}. Since $\L_\sharp^*\mu_\rho=0$ and $\<\L\varphi,\mu_\rho\>=\L\varphi(\rho)$, a necessary condition for \eqref{eq:phi2} is that
\begin{equation}\label{eq:phi2int}
\L\varphi(\rho)=\<\L_\flat\varphi_1,\mu_\rho\>.
\end{equation}
If \eqref{eq:phi2int} is satisfied, then we set
\begin{equation}\label{eq:phi22}
\varphi_2(f,\e)=\int_0^\infty\left(\E_{(f,\e)}\L_\flat\varphi_1(f_t,E_t)-\<\L_\flat\varphi_1,\mu_\rho\>\right) dt.
\end{equation}
The equation~\eqref{eq:phi2int} gives the limit generator $\L$. Since $f\mapsto H(f,\e)$, defined in \eqref{varphi1}, is linear, we have
\begin{align}
\L_\flat\varphi_1(f,\e)&=-(\div_x(vf),D_f\varphi_1(f,\e))\nonumber\\
&=(\div_x[H(\div_x(vf),\e)],D_\rho\varphi(\rho))+D^2_\rho\varphi(\rho)\cdot(\div_x(H(f,\e)),\div_x(J(f))),\label{Lbvarphi1}
\end{align}
and thus
\begin{equation}\label{limL}
\L\varphi(\rho)=(\<\psi,\mu_\rho\>,D_\rho\varphi(\rho))+\int_{E\times F} D^2_\rho\varphi(\rho)\cdot(\div_x(H(f,\e)),\div_x(J(f)))d\mu_\rho(f,\e),
\end{equation}
where $\psi(f,\e)=\div_x(H(\div_x(vf),\e)$. Let us compute the first term in the right-hand side of \eqref{limL}. Using \eqref{varphi1}, we have
\begin{equation}\label{psi12}
\psi(f,\e)=D^2_x\ps K(f)+\div_x\left[R_0(\e)\div_x(J(f))\right].
\end{equation}
The part $\<D^2_x\ps K(f),\mu_\rho\>=D^2_x\ps\left[\rho \E K(\bar{M}_0)\right]$ is given by \eqref{Kstar} below.

\begin{lemma} Let $\bar{M}^\LB_t$ and $\bar{M}^\FP_t$ be defined by \eqref{MbarLB} and \eqref{MbarFP} respectively. The expectation of the second moment of $\bar{M}_0$ is
\begin{align}
\E\left[K(\bar{M}_0)\right]=&K+\frac{b}{2}\E\left[\bar{E}(0)\osym R_1(\bar{E}(0))\right], \label{Kstar}
\end{align}
where $b^\LB=2$ and $b^\FP=1$.
\label{lem:Kfstar}\end{lemma}

\begin{proof}[Proof of Lemma~\ref{lem:Kfstar}] We compute, using \eqref{Moments},
\begin{align*}
K(\bar{M}_0^\LB)
=\int_{-\infty}^0 e^{\sigma}
\left(K+\left[\int_\sigma^0 \bar{E}(r)dr\right]^{\otimes 2}\right) d\sigma.
\end{align*}
This gives
\begin{align*}
\E\left[K(\bar{M}_0^\LB)\right]
=&K+\int_{-\infty}^0 e^{\sigma}\int_\sigma^0\int_\sigma^0\bar{\Gamma}(r-s)dr ds d\sigma,
\end{align*}
where $\bar{\Gamma}(t)$ is the covariance of $(\bar{E}(t))$ (see \eqref{Gammae}). We have also
$$
\int_\sigma^0\int_\sigma^0\bar{\Gamma}(r-s)dr ds =\int_\sigma^0 (r-\sigma)[\bar{\Gamma}(r)+\bar{\Gamma}(-r)] dr.
$$
Two successive integration by parts and \eqref{GammaGammacheck} give then \eqref{Kstar}.
Similarly, we have by \eqref{MbarFP} and \eqref{Moments},
\begin{align*}
K(\bar{M}_0^\FP)
=K+\left[\int_{-\infty}^0 e^{\sigma} \bar{E}(\sigma)d\sigma\right]^{\otimes 2}.
\end{align*}
To conclude to \eqref{Kstar}, we use the following Lemma~\ref{lem:sympos}.
\end{proof}

\begin{lemma} For $\delta>0$, we have
\begin{equation}\label{eq:sympos}
\E\left[R_\delta(\bar{E}(0))\osym \bar{E}(0)\right]=2\delta\E\left[\int_{-\infty}^0 e^{\delta\sigma} \bar{E}(\sigma) d\sigma\right]^{\otimes 2}.
\end{equation}
In particular, when $\delta\geq 0$, $\E\left[R_\delta(\bar{E}(0))\osym \bar{E}(0)\right]$ is a non-negative symmetric matrix.
\label{lem:sympos}\end{lemma}

\begin{proof}[Proof of Lemma~\ref{lem:sympos}] We compute 
\begin{align}
\E\left[\int_{-\infty}^0 e^{\delta\sigma} \bar{E}(\sigma) d\sigma\right]^{\otimes 2}=& \int_{-\infty}^0\int_{-\infty}^0  e^{\delta(\sigma+s)}\E[\bar{E}(s)\otimes \bar{E}(\sigma)]d\sigma ds\nonumber\\
=&\int_{-\infty}^0\int_{\sigma=-\infty}^s  e^{\delta(\sigma+s)}\E[\bar{E}(s)\osym\bar{E}(\sigma)]d\sigma ds
\end{align}
Then we set $\sigma'=\sigma-s$, and some standard rearrangements and computations give the formula \eqref{eq:sympos}. It is clear then that the left-hand side of \eqref{eq:sympos} is a non-negative matrix when $\delta>0$. This is also true for $\delta=0$ by continuity.
\end{proof}

\begin{remark} Using \eqref{Jstar}, we obtain
\begin{align}
\E\left[R_0(\bar{E}_0)\osym J(\bar{M}_0)\right]
&=\int_{-\infty}^0 e^{\sigma}\E\left[R_0(\bar{E}(0))\osym \bar{E}(\sigma)\right] d\sigma\nonumber\\
&=\E\left[R_1R_0(\bar{E}(0))\osym \bar{E}(0)\right].\label{AppJstar}
\end{align}
\end{remark}

To identify the contribution of the second part in \eqref{psi12}, we adapt \eqref{AppJstar} to get
\begin{align*}
\<\div_x\left[R_0(\e)\div_x(J(f))\right],\mu_\rho\>=&
\partial_{x_j}\int_{-\infty}^0 e^\sigma \E\left[R_0(\bar{E}_j(0))\partial_{x_i}\left(\rho \bar{E}_i(\sigma)\right) \right]\\
=&\partial_{x_j}\E\left[R_1R_0(\bar{E}_j(0))\partial_{x_i}\left(\rho \bar{E}_i(0)\right) \right].
\end{align*}
The first-order part in \eqref{limL} is therefore $(\<\psi,\mu_\rho\>,D_\rho\varphi(\rho))$, with 
\begin{equation}\label{limLfirstpre}
\<\psi,\mu_\rho\>=D^2_x\ps\left[ \rho\left(K+\frac{b}{2}\E\left[\bar{E}(0)\osym R_1(\bar{E}(0))\right]\right) \right]
+\div_x\E\left[R_1R_0(\bar{E}(0))\div_x\left(\rho \bar{E}(0)\right)\right].
\end{equation}
This can be rewritten as
\begin{equation}\label{limLfirst}
\<\psi,\mu_\rho\>=\div_x(K_\sharp\nabla_x\rho+\Theta\rho),
\end{equation}
where $K_\sharp$ and $\Theta$ are given in \eqref{Ksharp} and \eqref{PSI} respectively. 
To compute the second-order part in \eqref{limL}, we have two terms to consider: $\<J(f)\otimes J(f),\mu_\rho\>$ and $\<R_0(\e)\otimes J(f),\mu_\rho\>$. We have already established
$$
\<R_0(\e)\otimes J(f),\mu_\rho\>=\E\left[R_1R_0(\bar{E}(0))\otimes (\rho \bar{E}(0))\right].
$$
By \eqref{Jstar} and \eqref{eq:sympos}, we have also
$$
\<J(f)\otimes J(f),\mu_\rho\>=\E\left[(\rho R_1(\bar{E}(0)))\otimes (\rho \bar{E}(0))\right].
$$
It follows by the resolvent identity $R_1R_0=R_0-R_1$ that
\begin{multline}\label{limLsecond}
\int_{E\times F} D^2_\rho\varphi(\rho)\cdot(\div_x(H(f)),\div_x(J(f)))d\mu_\rho(f,\e)\\
=\E D^2_\rho\varphi(\rho)\cdot(\div_x(\rho R_0(\bar{E}(0))),\div_x(\rho \bar{E}(0))).
\end{multline}
To sum up, we find the following expression for the limit generator $\L$:
\begin{equation}\label{limitLfinal}
\L\varphi(\rho)=(\div_x(K_\sharp\nabla_x\rho+\Theta\rho),D_\rho\varphi(\rho))+\E D^2_\rho\varphi(\rho)\cdot(\div_x(\rho R_0(\bar{E}(0))),\div_x(\rho \bar{E}(0))).
\end{equation}

\begin{remark} Lemma~\ref{lem:sympos} has the following consequences, that we record here, in relation with Remark~\ref{rk:enhanced1}. Let us fix $\delta>0$ first. Assume that there is a non-trivial vector $q\in\R^N$ in the kernel of the matrix $\E\left[R_\delta(\bar{E}(0))\osym \bar{E}(0)\right]$. By \eqref{eq:sympos}, we obtain
\begin{equation}\label{trivenhanced1}
\int_{-\infty}^0 e^{\delta\sigma} \bar{m}(\sigma) d\sigma=0,\; a.s.,
\end{equation}
where $\bar{m}_t=\bar{E}_t\cdot q$. Since $(\bar{m}_t)$ is stationary, \eqref{trivenhanced1} implies that $(\bar{m}_t)$ is trivial. Indeed, \eqref{trivenhanced1} remains true if we replace $\bar{m}(\sigma)$ by $\bar{m}(\sigma+s)$ where $s$ is arbitrary in $\R$. A change of variable then shows that, almost-surely,
\[
V(s):=\int_{-\infty}^s e^{\delta(\sigma-s)} \bar{m}(\sigma) d\sigma=0
\]
Since $V'(s)=-\delta V(s)+\bar{m}(s)$, we conclude that $\bar{m}(s)=0$ indeed. Another way to reach this conclusion is to observe that, stated quite informally, we have the formula
\begin{equation}\label{GammaEnhanced}
\E\left[R_\delta\theta(\bar{E}(0))\theta(\bar{E}(0))\right]=\delta\int_F |\varphi|^2 d\nu+ \mathcal{E}(\varphi),
\end{equation}
where $\theta$ is a given a test-function, $\varphi=R_\delta\theta$, and 
\[
\mathcal{E}(\varphi):=\frac12\int_F (A|\varphi|^2-2\varphi A\varphi) d\nu
\]
is the Dirichlet form associated to $(\bar{E}_t)$. Equation~\eqref{GammaEnhanced} is a consequence of the identities
\[
\varphi=R_\delta\theta,\quad \theta=\delta\varphi-A\varphi,\quad \mathcal{E}(\varphi)=-\int_F \varphi A\varphi d\nu.
\]
We apply \eqref{GammaEnhanced} to the test-function $\theta(\e)=\e\cdot q$. When $\delta=0$, \eqref{GammaEnhanced} degenerates. We see however that, if $\e\mapsto\e\cdot q$ is not in the kernel of the Dirichlet form, then $q$ cannot be in the kernel of the matrix $\E\left[R_0(\bar{E}(0))\osym \bar{E}(0)\right]$. 
\label{rk:enhanced2}\end{remark}
\subsubsection{First and second correctors}\label{sec:phi12}

Recall (see \eqref{barJm}, \eqref{defG}) that 
$$
\bar{J}_m(f)=\iint_{\T^N\times\R^N}|v|^m f(x,v)dx dv,\quad G_m=\left\{ f\in L^1(\T^N\times\R^N); \bar{J}_m(f)<+\infty\right\}.
$$
Let us introduce the following notation. We write $a\lesssim b$ with the meaning that $a\leq C b$, where the constant $C$ may depend on $\mathtt{R}$ (see \eqref{BallR}), on $C^0_\mathtt{R}$ (see \eqref{AR0}), on various irrelevant constants, and on the dimension $N$.

\begin{proposition} Let $\varphi$ be of the form~\eqref{sepclass}, with $\xi\in C^3(\T^N)$ and $\psi$ a Lipschitz function of class $C^3$ on $\R$ such that the derivatives $\psi^{(j)}$, $j\in\{1,2,3\}$ are bounded. Let $\varphi_1$, $\varphi_2$ be the correctors defined by \eqref{eq:phi1}, \eqref{eq:phi22} respectively. Then the functions $\varphi_1$, $\varphi_2$ satisfy $\L_\sharp\varphi_i(f,\e)<+\infty$, $\L_\flat\varphi_i(f,\e)<+\infty$ for all $f\in G_3$, $\e\in F$ and are in the domain of $\L^\eps$. We have the estimates
\begin{equation}\label{boundphi1}
|\varphi_1(f,\e)|\lesssim \|\psi'\|_{C_b(\R)}\|\xi\|_{C^1(\T^N)}(\bar{J}_0(f)+\bar{J}_1(f)),
\end{equation}
and
\begin{equation}\label{boundLbphi1}
|\L_\flat\varphi_1(f,\e)|\lesssim \|\psi'\|_{C^1_b(\R)}\|\xi\|_{C^2(\T^N)}^2(|\bar{J}_0(f)|^2+|\bar{J}_2(f)|^2),
\end{equation}
on $\varphi_1$ and the following estimates on $\varphi_2$:
\begin{equation}
|\varphi_2(f,\e)|\lesssim \|\psi'\|_{C^1_b(\R)}\|\xi\|_{C^2(\T^N)}^2(|\bar{J}_0(f)|^2+|\bar{J}_2(f)|^2),\label{boundphi2}
\end{equation}
and
\begin{equation}
|\L_\flat\varphi_2(f,\e)|\lesssim \|\psi'\|_{C^2_b(\R)}\|\xi\|_{C^3(\T^N)}^3(|\bar{J}_0(f)|^3+|\bar{J}_3(f)|^3),\label{boundLbphi2}
\end{equation}
for all $f\in G_3$, for all $\e\in F$ with $\|\e\|_F\leq\mathtt{R}$. The estimate
\begin{equation}\label{boundLphi}
|\L\varphi(\rho)|\lesssim \|\psi'\|_{C_b(\R)}\|\xi\|_{C^2(\T^N)}^2\|\rho\|_{L^1(\T^N)}^2
\end{equation}
is also satisfied for all $\rho\in L^1(\T^N)$.
\label{prop:firstsecondcorrector}\end{proposition}

\begin{proof}[Proof of Proposition~\ref{prop:firstsecondcorrector}] For $\varphi$ as in \eqref{sepclass}, the formula \eqref{varphi1}, \eqref{Lbvarphi1} read
\begin{equation}\label{varphi1seclass}
\varphi_1(f,\e)=\psi'(\dual{\rho}{\xi})\dual{H(f,\e)}{\nabla_x\xi},\quad H(f,\e)=J(f)+\rho(f)R_0(\e),
\end{equation}
and
\begin{multline}\label{Lbvarphi1sepclass}
\L_\flat\varphi_1(f,\e)=\psi'(\dual{\rho}{\xi})\left[\dual{K_{ij}(f)}{\partial^2_{x_i x_j}\xi}+\dual{J_i(f)}{\partial_{x_i}(R_0(\e_j) \partial_{x_j}\xi)}\right]\\
+\psi''(\dual{\rho}{\xi})\dual{H(f,\e)}{\nabla_x\xi}\dual{J(f)}{\nabla_x\xi},
\end{multline}
respectively. The two estimates \eqref{boundphi1}, \eqref{boundLbphi1} then follow from the bounds \eqref{BallR}, \eqref{RoE} on $\bar{E}_t$ and $R_0(\e)$. The formula \eqref{eq:phi2int} for $\L\varphi$ and \eqref{boundLbphi1} then give \eqref{boundLphi}. Let us focus on the estimate \eqref{boundphi2} on $|\varphi_2(f,\e)|$ now. For simplicity, let us denote by $\psi',\psi'',\ldots$ the derivatives of $\psi$ evaluated at the point $\<\rho,\xi\>$. We start from the formula \eqref{eq:phi22}, which gives 
\begin{equation}\label{varphi2lin}
\varphi_2(f,\e)=\int_0^\infty\E_{(f,\e)}\left[\L_\flat\varphi_1(f_t,E_t)-\<\L_\flat\varphi_1,\mu_\rho\>\right] dt,
\end{equation}
where $f_t$ is obtained either by \eqref{fLBEEstar} or \eqref{fFPEEstar} with $s=0$.
Consider the $\LB$-case. There are two terms in $f_t$ and three terms in $\L_\flat\varphi_1$, which makes at least six terms to consider. We find more than six terms actually, because of the translations in $v$. Consider the first term in \eqref{fLBEEstar}. By \eqref{Moments}, and for 
$$
w_t:=\int_0^t E_s(\e) ds,
$$ 
we have
\begin{equation*}
K(f(\cdot-w_t))=K(f)+J(f)\osym w_t+\rho(f) w_t^{\otimes 2},\quad
J(f(\cdot-w_t))=J(f)+\rho(f) w_t.
\end{equation*}
In \eqref{Lbvarphi1sepclass}-\eqref{varphi2lin}, and regarding the linear terms with factor $\psi'$, this gives the contributions 
\begin{align*}
\Phi_{2,a}=\psi'\int_{\T^N}\int_0^\infty e^{-t}\E\left[K(f)+J(f)\osym w_t+\rho(f) w_t^{\otimes 2}\right]\ps D^2_x\xi dt dx,
\end{align*}
and
\begin{align*}
\Phi_{2,b}=\psi'\int_{\T^N}\int_0^\infty e^{-t}\E\left[(J(f)+\rho(f) w_t)\cdot\nabla_x[R_0(E_t(\e))\cdot\nabla_x\xi]\right]dt dx.
\end{align*}
Using the bound $\|w_t\|_F\leq t\sup_{s\in[0,t]}\|E_s(\e)\|_F$ and \eqref{BallR}, \eqref{RoE}, we have
$$
|\Phi_{2,a}|,|\Phi_{2,b}|\lesssim \|\psi'\|_{C_b(\R)}\|\xi\|_{C^2(\T^N)}(\bar{J}_0(f)+\bar{J}_1(f)+\bar{J}_2(f)).
$$
Since $\bar{J}_1(f)\leq\frac12\bar{J}_0(f)+\frac12\bar{J}_2(f) $, this gives us a bound by $\|\psi'\|_{C_b(\R)}\|\xi\|_{C^2(\T^N)}(\bar{J}_0(f)+\bar{J}_2(f))$. Using \eqref{Moments} again, and still regarding the linear terms with factor $\psi'$ only, we see that the second term in the expansion~\eqref{fLBEEstar} of $f_t^\LB$ has the contributions 
$$
\Phi_{2,c}=\psi'\int_{\T^N}\int_0^\infty (\theta_c(t)-\theta_c(+\infty))dt dx,\quad \Phi_{2,d}=\psi'\int_{\T^N}\int_0^\infty (\theta_d(t)-\theta_d(+\infty))dt dx,
$$
where
\begin{align}
\theta_c(t)=&\rho(f)\int_0^t e^{-(t-\sigma)}\left[K+\E\left(\int_\sigma^t E_s(\e)ds\right)^{\otimes 2}\right]\ps D^2_x\xi d\sigma,\label{thetac}\\
\theta_d(t)=&\rho(f)\int_0^t e^{-(t-\sigma)}\E\left[\int_\sigma^t E_s(\e)ds \cdot\nabla_x[R_0(E_t(\e))\cdot\nabla_x\xi]\right]d\sigma.\label{thetad}
\end{align}
By standard manipulations on the integrals in \eqref{thetac}, we have
$$
\theta_c(t)=\rho(f)(1-e^{-t})K\ps D^2_x\xi+2\rho(f)\int_0^t e^{-\sigma}\int_0^\sigma\int_r^\sigma\Gamma_\e(t-r,t-s)\ps D^2_x\xi ds dr d\sigma,
$$
where the covariance $\Gamma_\e$ is defined by \eqref{Gammae}. The most delicate term to estimate in $\Phi_{2,c}$ is 
$$
\Phi_{2,c}^*=2\int_{\T^N}\rho(f) \int_0^\infty\int_0^t e^{-\sigma}\int_0^\sigma\int_r^\sigma\left[\Gamma_\e(t-r,t-s)-\bar{\Gamma}(s-r)\right]\ps D^2_x\xi ds dr d\sigma dt dx.
$$
The other terms are bounded by $\|\xi\|_{C^2(\T^N)}(\bar{J}_0(f)+\bar{J}_2(f))$ using \eqref{BallR}. Using also \eqref{MixGamma}, we have
\begin{align*}
|\Phi_{2,c}^*|\lesssim &\ 2\bar{J}_0(f)\|\xi\|_{C^2(\T^N)} \int_0^\infty\int_0^t e^{-\sigma}\int_0^\sigma\int_r^\sigma\gamma_\mathrm{mix}(t-s)ds dr d\sigma dt\\
\lesssim &\ 2\bar{J}_0(f)\|\xi\|_{C^2(\T^N)} \int_0^\infty\int_0^t s(e^{-s}-e^{-t})\gamma_\mathrm{mix}(t-s)ds dt.
\end{align*}
Neglecting the term $-e^{-t}$ and using \eqref{normalizegammamix} gives a bound 
$
|\Phi_{2,c}^*|\lesssim 2\bar{J}_0(f)\|\xi\|_{C^2(\T^N)}.
$
We have also
$$
\theta_d(t)=\rho(f)\int_0^t e^{-\sigma}\int_0^\sigma \E\left[E_{t-s}(\e) \cdot\nabla_x[R_0(E_t(\e))\cdot\nabla_x\xi]\right] ds d\sigma.
$$
Conditioning on $\G_{t-s}$, we see that 
\begin{equation}\label{GOKOK}
\E\left[E_{t-s}(\e) \otimes R_0(E_t(\e))\right]=\mathtt{P}_{t-s}\left[\psi\otimes \mathtt{P}_{s}R_0\psi\right](\e),\quad \psi(\e)=\e.
\end{equation}
Indeed, given some continuous and bounded functions $\varphi,\theta\colon F\to\R$, the Markov property gives
\[
\E\left[\varphi(E_{t-s}(\e)) \otimes \theta(E_t(\e))|\G_{t-s}\right]=\varphi(E_{t-s}(\e))\otimes\mathtt{P}_s\theta(E_{t-s}(\e)).
\]
Taking expectation gives 
\begin{equation}\label{GOKOK0}
\E\left[\varphi(E_{t-s}(\e)) \otimes \theta(E_t(\e))\right]=\mathtt{P}_{t-s}\left[\varphi\otimes\mathtt{P}_s\theta\right](\e).
\end{equation}
The bound \eqref{BallR} allows us to extend \eqref{GOKOK0} to the case $\varphi=\psi$, $\theta=R_0\psi$ to establish \eqref{GOKOK}.
By \eqref{gammamixquad}, \eqref{RoE}, \eqref{normalizegammamix}, we obtain
\begin{align*}
\Bigg\|\int_0^\infty\int_0^t e^{-\sigma}\int_0^\sigma& \Big[\mathtt{P}_{t-s}\big[\psi\cdot\nabla_x (\mathtt{P}_{s}R_0\psi\cdot\nabla_x\xi)\big](\e)\\
&\hspace*{3cm}-\<\psi\cdot\nabla_x (P_sR_0\psi\cdot\nabla_x\xi),\nu\> \Big] ds d\sigma dt\Bigg\|_{C(\T^N)}\\
&\leq \mathtt{R}^2 \int_0^\infty\int_0^t e^{-\sigma}\int_0^\sigma\gamma_\mathrm{mix}(t-s)ds d\sigma dt\|\xi\|_{C^1(\T^N)}\leq \mathtt{R}^2 \|\xi\|_{C^1(\T^N)}.
\end{align*}
Using this bound, it is easy to prove that $|\Phi_{2,d}|\lesssim \|\xi\|_{C^2(\T^N)}\bar{J}_0(f)$. Let us look at the quadratic terms with the factor $\psi''$ now. There are two terms in \eqref{fLBEEstar}, so four terms $\Phi_{2,e},\ldots,\Phi_{2,h}$ to consider here. The first term in \eqref{fLBEEstar} has a factor $e^{-t}$, like in $\Phi_a$, $\Phi_b$. There is no contribution from $\<\L_\flat\varphi_1,\mu_\rho\>$ in $\Phi_{2,e},\Phi_{2,f},\Phi_{2,g}$ hence, and the convergence of the integral in \eqref{varphi2lin} is clear. Therefore, using the same arguments as above, we obtain the estimates
\begin{equation}\label{boundPhi2efg}
|\Phi_{2,e}|,|\Phi_{2,f}|,|\Phi_{2,g}|\lesssim\|\psi''\|_{C_b(\R)}\|\nabla_x\xi\|_{C^1(\T^N)}^2(|\bar{J}_0(f)|^2+|\bar{J}_1(f)|^2).
\end{equation} 
Let us illustrate this on the example of $\Phi_{2,g}$. We have
\begin{multline*}
\Phi_{2,g}=\psi''\int_0^\infty e^{-t}\int_0^t e^{-(t-\sigma)} \E\left[\int_\sigma^t\<\rho(f) E_r(\e),\nabla_x\xi\>_{L^2(\T^N)}dr \right.\\
\left.\times\<J(f)+\rho(f)\int_0^t E_s(\e) ds+\rho(f)R_0(E_t),\nabla_x\xi\>_{L^2(\T^N)}\right],
\end{multline*}
which gives \eqref{boundPhi2efg}. The last term $\Phi_{2,h}$ is
$$
\Phi_{2,h}=\psi''\int_0^\infty(\theta_h(t)-\theta_h(+\infty))dt,
$$
where
\begin{multline*}
\theta_h(t)=\E\int_0^t\int_\sigma^t\int_0^t\int_{\sigma'}^t e^{-(t-\sigma)}e^{-(t-\sigma')}\<\rho(f)E_s(\e),\nabla_x\xi\>_{L^2(\T^N)}\\
\times \<\rho(f)E_{s'}(\e)+c(t)^{-1}\rho(f)R_0(E_t(\e)),\nabla_x\xi\>_{L^2(\T^N)} ds'  ds d\sigma' d\sigma .
\end{multline*}
The coefficient $c(t)$ is
$$
c(t)=\int_0^t\int_{\sigma'}^t e^{-(t-\sigma')} ds' d\sigma'=\int_0^t\sigma e^{-\sigma}=1-(t+1)e^{-t}.
$$
The technique used to estimate the terms $\Phi_{2,c}$ and $\Phi_{2,d}$ applies here to give
$$
|\Phi_{2,h}|\lesssim\|\psi''\|_{C_b(\R)}\|\nabla_x\xi\|_{C^1(\T^N)}^2|\bar{J}_0(f)|^2.
$$
This concludes the estimate on $\varphi_2$ in the $\LB$-case. The estimate on $\varphi_2$ in the $\FP$-case is obtained by the same arguments. This follows from the expressions for $K(f_t)$, $J(f_t)$, which involve various terms, similar to those estimated in the $\LB$-case. For example, a careful computation based on \eqref{fFPEEstar} and \eqref{Moments} gives
\begin{align*}
K(f_t^\FP)=&\rho(f)\left[(1-e^{-2t})K+\left(\int_0^t e^{-(t-\sigma)}E_\sigma(\e)d\sigma\right)^{\otimes 2}\right]+e^{-2t}K(f)\\&+e^{-t}\left[\int_0^t e^{-(t-\sigma)}E_\sigma(\e)d\sigma\otimes J(f)+J(f)\otimes\int_0^t e^{-(t-\sigma)}E_\sigma(\e)d\sigma\right].
\end{align*}
A comparable expansion for $J(f_t^\FP)$ gives the result, like in the $\LB$-case. Using~\eqref{AR0}, a careful study of the terms composing $\varphi_2$ shows that $\varphi_1$ and $\varphi_2$ are of the form \eqref{eq:admtest} with some $\xi_i$ as in Remark~\ref{rk:admtest}. By Proposition~\ref{prop:admtest}, we deduce that $\L_\sharp\varphi_i(f,\e)<+\infty$, $\L_\flat\varphi_i(f,\e)<+\infty$ for all $f\in G_3$, $\e\in F$ and that $\varphi_1$ and $\varphi_2$ are in the domain of $\L^\eps$.

There remains to prove \eqref{boundLbphi2}. Compared to the development of $\varphi_2$, when computing $\L_\flat\varphi_2$, still more terms appear, which combine the derivatives of $\psi$ up to the order three. However, all the questions of convergence of the integrals with respect to $t$ have been dealt with in the estimate of $\varphi_2$. Although lengthy, it is not problematic to prove \eqref{boundLbphi2}: we do not expound that part thus.
\end{proof}

\begin{remark}[Linear test function] In Section~\ref{sec:tight}, we apply Proposition~\ref{prop:firstsecondcorrector} to a linear test-function $\varphi(\rho)=\<\rho,\xi\>_{L^2(\T^N)}$, which means $\psi'=1$, $\psi''=0$. In that case, the bounds on the first corrector is a little bit simpler: we have
\begin{equation}\label{boundphi1lin}
|\varphi_1(f,\e)|\lesssim \|\xi\|_{C^1(\T^N)}(\bar{J}_0(f)+\bar{J}_1(f)),
\end{equation}
and
\begin{equation}\label{boundLbphi1lin}
|\L_\flat\varphi_1(f,\e)|\lesssim \|\xi\|_{C^2(\T^N)}(\bar{J}_0(f)+\bar{J}_2(f)),
\end{equation}
for all $f\in G$, for all $\e\in F$ with $\|\e\|_F\leq\mathtt{R}$.
\end{remark}

By Theorem~\ref{th:MarkovPty}, Remark~\ref{rk:admtest2} and Theorem~\ref{th:XMartingaleE}, we obtain the following corollary to Proposition~\ref{prop:firstsecondcorrector}.

\begin{corollary} Let $\varphi$ be of the form~\eqref{sepclass}, with $\xi\in C^3(\T^N)$ and $\psi$ a Lipschitz function of class $C^3$ on $\R$ such that the derivatives $\psi^{(j)}$, $j\in\{1,2,3\}$ are bounded. Let $\varphi_1$, $\varphi_2$ be the correctors defined by \eqref{eq:phi1}, \eqref{eq:phi22} respectively. Let $\theta$ be the correction of $\varphi$ at order $0$, $1$ or $2$:
$$
\theta\in\{\varphi,\varphi+\eps\varphi_1,\varphi+\eps\varphi_1+\eps^2\varphi_2\}.
$$
Then
\begin{equation}\label{martingaleepstheta}
M^\eps_\theta(t):=\theta(f^\eps_t,\bar{E}^\eps_t)-\theta(f_\mathrm{in},\bar{E}_0)-\int_0^t\L^\eps\theta(f^\eps_s,\bar{E}^\eps_s)ds
\end{equation}
and
\begin{equation*}
|M^\eps_\theta(t)|^2-\int_0^t\left[\L^\eps|\theta|^2-2\theta\L^\eps\theta\right](f^\eps(s),\bar{E}^\eps(s))ds
\end{equation*}
are $(\G_{t/\eps^2})$-martingales.
%
\label{cor:firstsecondcorrector}\end{corollary}

\subsection{Bounds on the moments}\label{sec:bounds}

Recall that $\bar{J}_m(f)$ denotes the $m$-th moment of $f$ (see \eqref{barJm}) and that $G_m$ is the space of functions $f\in L^1(\T^N\times\R^N)$ such that $\bar{J}_m(f)<+\infty$.

\begin{proposition}  Let $f^\eps_0\in G_m$. Let $(f^\eps_t)$ be the unique mild solution to \eqref{Eqrescaled} on $[0,T]$ given by Proposition~\ref{prop:CYLB} or \ref{prop:CYFP}. Then, for all $m\in\N$, almost-surely, for all $t\geq 0$, 
\begin{equation}\label{eq:MomentBound}
\bar{J}_m(f^\eps_t)\leq C(\mathtt{R},m,t)\left[\bar{J}_m(f^\eps_0)+\bar{J}_0(f^\eps_0)\right],
\end{equation}
where $C(\mathtt{R},m,t)$ is a constant which is bounded for $t$ in a bounded set.
\label{prop:MomentBound}\end{proposition}

\begin{proof}[Proof of Proposition~\ref{prop:MomentBound}] By density, we can assume that $f_\mathrm{in}\in W^{2,1}(\T^N\times\R^N)$. We can also replace $v\mapsto |v|^m$ by $v\mapsto |v|^m\chi_\eta(v)$, where $\chi_\eta$ is a function with compact support which converges pointwise to $1$ when $\eta\to0$. By the results of propagation of regularity given in Proposition~\ref{prop:CYLB} and Proposition~\ref{prop:CYFP}, the following computations are licit then. For simplicity, we take directly $\chi\equiv 1$. First, we have
\begin{equation}\label{Lmoment1}
\frac{d\;}{dt}\bar{J}_{2m}(f^\eps_t)=\frac{1}{\eps^2}\left[\bar{J}_{2m}(Qf^\eps_t)+2m\iint_{\T^N\times\R^N}|v|^{2(m-1)} v\cdot \bar{E}^\eps_t f^\eps_t(x,v)dx dv\right].
\end{equation}
If $m=0$, then, for all $t\geq 0$, almost-surely, $\bar{J}_{0}(f^\eps_t)=\bar{J}_{0}(f^\eps_0)$ since the equation is conservative. If $m>0$, then we use the following inequality (which is a consequence of Young's inequality)
$$
2m|v|^{2m-1}\leq \frac{1}{2\mathtt{R}} |v|^{2m}+[2\mathtt{R}(2m-1)]^{2m-1},
$$
to infer, by \eqref{Lmoment1} and \eqref{RE}, that
\begin{equation*}
\frac{d\;}{dt}\bar{J}_{2m}(f^\eps_t)\leq\frac{1}{\eps^2}\left[\bar{J}_{2m}(Qf^\eps_t)+\frac12 \bar{J}_{2m}(f^\eps_t)+\mathtt{R}[2\mathtt{R}(2m-1)]^{2m-1}\bar{J}_0(f^\eps_t)\right].
\end{equation*}
We have, in the case $Q=\QLB$, 
$$
\bar{J}_{2m}(\QLB f)=\bar{J}_{2m}(M)\bar{J}_0(f)-\bar{J}_{2m}(f).
$$ 
If $Q=\QFP$, then
\begin{align*}
\bar{J}_{2m}(\QFP f)=&-2m\iint_{\T^N\times\R^N}|v|^{2(m-1)} v\cdot (\nabla_v f(x,v)+vf(x,v))dx dv\\
=&(N+2(m-1))\bar{J}_{2(m-1)}(f)-2m\bar{J}_{2m}(f).
\end{align*}
In the first case $Q=\QLB$, we obtain
$$
\bar{J}_{2m}(f^\eps_t)\leq e^{-\frac{t}{2\eps^2}}\bar{J}_{2m}(f^\eps_0)+2(1-e^{-\frac{t}{2\eps^2}})\left[\bar{J}_{2m}(M)+\mathtt{R}[2\mathtt{R}(2m-1)]^{2m-1}\right]\bar{J}_0(f^\eps_0).
$$
This gives \eqref{eq:MomentBound}. If $Q=\QFP$, we conclude similarly by a recursive argument on $m$.
\end{proof}

\subsection{Tightness}\label{sec:tight}
For $\sigma>0$, we denote by $H^{-\sigma}(\T^N)$ the dual space of $H^\sigma(\T^N)$. Let $J_1^\sigma=(\mathrm{Id}-\Delta_x)^{-\sigma}$. In the standard Fourier basis $(w_k)$ of $L^2(\T^N)$, $J_1^\sigma$ is given by 
$$
J_1^\sigma w_k=(1+\lambda_k)^{-\sigma} w_k,\quad\lambda_k=4\pi^2|k|^2,\quad w_k(x)=\exp(2\pi i k\cdot x).
$$ 
As $J_1^{\sigma/2}$ is an isometry $L^2(\T^N)\to H^\sigma(\T^N)$, the norm on $H^{-\sigma}(\T^N)$ is
\begin{equation}\label{Hminus}
\|\Lambda\|_{H^{-\sigma}(\T^N)}=\left[\sum_{k\in\Z^d}|\<\Lambda,J_1^{\sigma/2} w_k\>_{L^2(\T^N)}|^2\right]^{1/2}.
\end{equation}

\begin{proposition}[Tightness] Let $f^\eps_0\in G_3$. Let $(f^\eps_t)$ be the unique mild solution to \eqref{Eqrescaled} on $[0,T]$ given by Proposition~\ref{prop:CYLB} or \ref{prop:CYFP}. Then $(\rho^\eps_t)_{t\in[0,T]}$ is tight in the space $C([0,T];H^{-1}(\T^N))$.
\label{prop:tight}\end{proposition}

\begin{proof}[Proof of Proposition~\ref{prop:tight}]  Let us introduce the decomposition
\begin{equation}\label{decrhothetazeta}
\rho^\eps=\theta^\eps+\zeta^\eps,\quad\theta^\eps=\eps\div_x(J(f^\eps)+\rho(f^\eps)R_0(\bar{E}^\eps_t)).
\end{equation}
Note that, contrary to $\rho^\eps$, which has continuous trajectories, $\theta^\eps$ and $\zeta^\eps$ are, \textit{a priori}, c\`adl\`ag processes, just like $\bar{E}^\eps$. We show first that $\rho^\eps$ is close to $\zeta^\eps$ in the norm of $C([0,T];H^{-1}(\T^N))$ and then prove in a second step that $(\zeta^\eps)$ is tight in the Skorokhod space $D([0,T];H^{-1}(\T^N))$. In the third last step, we show that $(\rho^\eps_t)_{t\in[0,T]}$ is tight in $C([0,T];H^{-1}(\T^N))$.\medskip

\textbf{Step 1. $\rho^\eps$ is close to $\zeta^\eps$.} This is a straightforward consequence of the bound on the moments \eqref{eq:MomentBound}. Let us extend the notation $a\lesssim b$ to denote the inequality $a\leq C b$, where the factor $C$ may depend on $\mathtt{R}$, on $C^0_\mathtt{R}$, on $N$ and also on $\sup_{0<\eps<1}\bar{J}_m(f^\eps_0)$ for $m=0,\ldots,3$ and on $T$. Note that $C$ should not depend on $\eps$, nor on $\omega$. Then, by \eqref{eq:MomentBound}, we have $\sup_{t\in[0,T]}\|\theta^\eps_t\|_{H^{-1}(\T^N)}\lesssim\eps$. \medskip

\textbf{Step 2. $(\zeta^\eps)$ is tight in $D([0,T];H^{-1}(\T^N))$.} The bound on the moments \eqref{eq:MomentBound} shows that 
$\sup_{t\in[0,T]}\|\rho^\eps_t\|_{L^2(\T^N)}$ and $\sup_{t\in[0,T]}\|\theta^\eps_t\|_{H^{-1}(\T^N)}$ are almost-surely bounded. Since $\zeta^\eps=\rho^\eps-\theta^\eps$, the quantity  $\sup_{t\in[0,T]}\|\zeta^\eps_t\|_{H^{-1}(\T^N)}$ is also almost-surely bounded. By \cite[Theorem~3.1]{Jakubowski86}, it is sufficient therefore to prove that, for all $\xi\in C^2(\T^N)$, the family of real-valued processes $\<\zeta^\eps,\xi\>$ is tight in $D([0,T])$. Let us fix such a $\xi$, and let us set $\varphi(\rho)=\<\rho,\xi\>$ and $\gamma^\eps=\<\zeta^\eps,\xi\>$. Denote by
\[
\varphi_{1}(f,\e)=\<J(f)+\rho(f)R_0(\e),\xi\>
\]
the first corrector associated to $\varphi$. To obtain an estimate on the time increments of $\gamma^\eps$, we introduce the perturbed test function $\varphi^\eps=\varphi+\eps\varphi_{1}$ and the martingale (see \eqref{martingaleepstheta})
\begin{equation}\label{Mzeta}
M^\eps(t)=\varphi^\eps(f^\eps(t),\bar{E}^\eps(t))-\varphi^\eps(f^\eps(0),\bar{E}^\eps(0))-\int_0^t\L^\eps\varphi^\eps(f^\eps(s),\bar{E}^\eps(s))ds.
\end{equation}
We have thus
\begin{equation}\label{gammaML}
\gamma^\eps_t=\int_0^t\L^\eps\varphi^\eps(f^\eps(\sigma),\bar{E}^\eps(\sigma))d\sigma+M^\eps(t).
\end{equation}
To prove that $(\gamma^\eps_t)$ is tight in $D([0,T])$, we will use the Aldous criterion, \cite[Theorem~4.5, p.356]{JacodShiryaev03}. Let $1>\theta>0$. Let $\tau_1,\tau_2$ be some $(\F^\eps_t)$-stopping times such that 
\begin{equation}\label{tau1tau2}
\tau_1\leq\tau_2\leq\tau_1+\theta ,\quad\tau_2\leq T,\mbox{ a.s.}
\end{equation}
By the Doob optional sampling theorem, we have
\[
\E\left[|M^\eps(\tau_2)-M^\eps(\tau_1)|^2\right]=\E\left[|M^\eps(\tau_2)|^2-|M^\eps(\tau_1)|^2\right].
\]
Let $(A^\eps_t)$ be defined by \eqref{defAt}, where $\L=\L^\eps$ and $\varphi=\varphi^\eps$. By Theorem~\ref{th:XMartingaleE}, $|M^\eps(t)|^2-A^\eps_t$ is a martingale. Consequently, 
\[
\E\left[|M^\eps(\tau_2)-M^\eps(\tau_1)|^2\right]=\E\left[A^\eps_{\tau_2}-A^\eps_{\tau_1}\right]\lesssim\theta,
\]
We also have
\[
\E\left|\int_{\tau_1}^{\tau_2}\L^\eps\varphi^\eps(f^\eps(s),\bar{E}^\eps(s))ds\right|^2\lesssim\theta^2.
\]
Using the decomposition \eqref{gammaML}, we conclude that the increments of $\gamma^\eps$ also satisfy the estimate
\[
\E\left[|\gamma^\eps_{\tau_2}-\gamma^\eps_{\tau_1}|^2\right]\lesssim\theta.
\]
By the Markov inequality, the Aldous criterion
\[
\lim_{\theta\to 0}\limsup_{\eps\in(0,1)}\sup_{\tau_1,\tau_2}\P(|\gamma^\eps_{\tau_2}-\gamma^\eps_{\tau_1}|>\eta)=0
\]
is satisfied for all $\eta>0$, (the sup on $\tau_1,\tau_2$ being the sup over the stopping times satisfying \eqref{tau1tau2}). This gives the desired conclusion.\medskip

\textbf{Step 3. $(\rho^\eps)$ is tight in $C([0,T];H^{-1}(\T^N))$.} Using Step 1. and \cite[Lemma~3.31 p.352]{JacodShiryaev03}, we deduce that $(X^\eps_t)$ is tight in $D([0,T];\R^d)$. Since $(X^\eps_t)$ is in $C([0,T];\R^d)$, it is actually tight in $C([0,T];\R^d)$. To establish this fact, it is sufficient to use the relation $w_\rho(\delta)\leq 2 w_\rho'(\delta)$ ($t\mapsto\rho(t)$ continuous) between the modulus of continuity of continuous functions and the modulus of continuity of c\`adl\`ag functions, see \cite[(12.10) p.~123]{BillingsleyBook}.
\end{proof}

\subsection{Convergence to the solution of a Martingale problem}\label{sec:cvDA}

Assume that the hypotheses of Proposition~\ref{prop:tight} are satisfied. Let $\eps_\N=\{\eps_n;n\in\N\}$, where $(\eps_n)\downarrow 0$. By the Skorokhod theorem \cite[p.~70]{BillingsleyBook}, there is a subset of $\eps_\N$, which we still denote by $\eps_\N$, a probability space $(\tilde{\Omega},\tilde{\F},\tilde{\P})$, some random variables $\{\tilde\rho^{\eps};\eps\in\eps_\N\}$, $\tilde{\rho}$ on $C([0,T];H^{-1}(\T^N))$, such that
\begin{enumerate}
\item for all $\eps\in\eps_\N$, the laws of $\rho^{\eps}$ and $\tilde\rho^{\eps}$ as $C([0,T];H^{-1}(\T^N))$-random variables coincide,
\item\label{item:cvtihht2} $\tilde{\P}$-a.s., $(\tilde\rho^{\eps})$ is converging to $\tilde\rho$ in $C([0,T];H^{-1}(\T^N))$ along $\eps_\N$.
\end{enumerate}
By lower semi-continuity, we have $\tilde{\P}$-a.s., for all $t\in[0,T]$, $\tilde{\rho}_t\in L^1(\T^N)$. Let $(\tilde{\F}_t)_{t\in[0,T]}$ be the natural filtration of $(\tilde{\rho}(t))_{t\in[0,T]}$. Our aim is to show that the process $(\tilde{\rho}(t))_{t\in[0,T]}$ is a solution of the martingale problem associated to the limit generator $\L$.

\begin{proposition}[Martingale] Let $\xi\in C^3(\T^N)$, and let $\varphi$ be of the form \eqref{sepclass}, where $\psi$ is a Lipschitz function of class $C^3$ on $\R$ such that the derivatives $\psi^{(j)}$, $j\in\{1,2,3\}$ are bounded.
Let $\L$ be the limit generator defined by \eqref{limitLfinal}. Then the process 
\begin{equation}\label{Mvarphitilde}
\tilde{M}_\varphi(t):=\varphi(\tilde{\rho}(t))-\varphi(\tilde{\rho}(0))-\int_0^t\L\varphi(\tilde{\rho}(s))ds
\end{equation}
is a continuous martingale with respect to $(\tilde{\F}_t)_{t\in[0,T]}$. Let 
\begin{equation}\label{defQrhoxi}
Q(\rho;\xi)=\E\left[\dual{\rho}{R_0(\bar{E}(0))\cdot\nabla_x\xi}\dual{\rho}{\bar{E}(0)\cdot\nabla_x\xi}\right].
\end{equation}
The quadratic variation of $(\tilde{M}_\varphi(t))$ has the expression
\begin{equation}\label{QvarMvarphitilde}
\<\tilde{M}_\varphi,\tilde{M}_\varphi\>_t=2\int_0^t |\psi'(\dual{\tilde{\rho}_s}{\xi})|^2 Q(\tilde{\rho}_s;\xi)ds,
\end{equation}
for all $t\in[0,T]$.
\label{prop:MartingalePbtilde}\end{proposition}

\begin{proof}[Proof of Proposition~\ref{prop:MartingalePbtilde}] Recall that $\L\varphi(\tilde{\rho}(s))$ is well defined by \eqref{boundLphi}. For $\varphi$ given by \eqref{sepclass} in the expression \eqref{limitLfinal} of the limit generator, we get the decomposition
\begin{multline}\label{limitLfinalsepclass}
\L\varphi(\rho)=\psi'(\dual{\rho}{\xi})\dual{\rho}{\div_x(K_\sharp\nabla_x\xi)+\Theta\cdot\nabla_x\xi}\\
+\psi''(\dual{\rho}{\xi})\E\left[\dual{\rho}{R_0(\bar{E}(0))\cdot\nabla_x\xi}\dual{\rho}{\bar{E}(0)\cdot\nabla_x\xi}\right].
\end{multline}
Since $\xi\in C^3(\T^N)$, we obtain that $\rho\mapsto\L\varphi(\rho)$ is continuous for the $H^{-1}(\T^N)$-topology, thereby showing that the process $(\tilde{M}_\varphi(t))$ is continuous. Let us prove now the martingale property. Let $0\leq s\leq t\leq T$. Let $0\leq t_1<\cdots<t_n\leq s$ and let $\Theta$ be a continuous and bounded function on $[H^{-1}(\T^N)]^n$. Note that $\tilde{\F}_s$ is generated by the random variables $\Theta(\tilde{\rho}(t_1),\ldots,\tilde{\rho}(t_n))$, for $n\in\N^*$, $(t_i)_{1,n}$ and $\Theta$ as above. Our aim is therefore to prove that
\begin{equation}\label{Martingaletilde}
\E\left[(\tilde{M}_\varphi(t)-\tilde{M}_\varphi(s))\Theta(\tilde{\rho}(t_1),\ldots,\tilde{\rho}(t_n))\right]=0.
\end{equation}
Let $\varphi^\eps=\varphi+\eps\varphi_1+\eps^2\varphi_2$ be the second order correction of $\varphi$, with $\varphi_1$ and $\varphi_2$ given by Proposition~\ref{prop:firstsecondcorrector}. We start from the identity (see \eqref{martingaleepstheta})
\begin{equation}\label{Martingaletildeeps}
\E\left[(M^\eps_\varphi(t)-M^\eps_\varphi(s))\Theta(\rho^\eps(t_1),\ldots,\rho^\eps(t_n))\right]=0,
\end{equation}
where
\begin{equation}\label{Mvarphitildeeps}
M^\eps_\varphi(t):=\varphi^\eps(f^\eps(t),\bar{E}^\eps_t)-\varphi^\eps(f_\mathrm{in},\bar{E}^\eps_0)-\int_0^t\L^\eps\varphi^\eps(f^\eps(s),\bar{E}^\eps_s)ds,
\end{equation}
Recall that $\L^\eps\varphi^\eps=\L\varphi+\eps\L_\flat\varphi_2$. By \eqref{Martingaletildeeps}, the estimates on the correctors (Proposition~\ref{prop:firstsecondcorrector}) 
and the uniform estimates on the moments of $(f^\eps_t)$ (Proposition~\ref{prop:MomentBound}), we have
\begin{align*}
\E\left[(X^\eps_\varphi(t)-X^\eps_\varphi(s))\Theta(\rho^\eps(t_1),\ldots,\rho^\eps(t_n))\right]=\mathcal{O}(\eps),
\end{align*}
where the process $(X^\eps_\varphi(t))$ is 
$$
X^\eps_\varphi(t)=\varphi(\rho^\eps(t))-\varphi(\rho_\mathrm{in})-\int_0^t\L\varphi(\rho^\eps(s))ds.
$$
By identities of the laws, it follows that 
\begin{align}
\tilde{\E}\left[\left(\varphi(\tilde{\rho}^\eps(t))-\varphi(\tilde{\rho}^\eps(s))-\int_s^t\L\varphi(\tilde{\rho}^\eps(s))ds\right)\Theta(\tilde{\rho}^\eps(t_1),\ldots,\tilde{\rho}^\eps(t_n))\right]=\mathcal{O}(\eps).\label{Martingaletilderhoeps}
\end{align}
We must examine the convergence of each terms in \eqref{Martingaletilderhoeps}. By a.s convergence of $(\tilde\rho^{\eps})$ in $C([0,T];H^{-1}(\T^N))$ along $\eps_\N$, we have
\begin{multline*}
\left[\varphi(\tilde{\rho}^\eps(t))-\int_0^t\L\varphi(\tilde{\rho}^\eps(s))ds\right]\Theta(\tilde{\rho}^\eps(t_1),\ldots,\tilde{\rho}^\eps(t_n))\\
\to \left[\varphi(\tilde{\rho}(t))-\int_0^t\L\varphi(\tilde{\rho}(s))ds\right]\Theta(\tilde{\rho}(t_1),\ldots,\tilde{\rho}(t_n))
\end{multline*}
almost-surely when $\eps\to 0$ along $\eps_\N$. Since $\Theta$ is bounded and $\varphi(\tilde{\rho}^\eps(t))$ and $\L\varphi(\tilde{\rho}^\eps(t))$ are a.s. bounded by a constant (a consequence of \eqref{eq:MomentBound}), we can apply the dominated convergence theorem. This gives \eqref{Martingaletilde}. Because $\tilde{M}_\varphi$ is continuous, the quadratic variation of $\tilde{M}_\varphi$ is the unique non-decreasing process $(A_t)$ such that $|\tilde{M}_\varphi(t)|^2-A_t$ is a martingale. Theorem \ref{th:XMartingaleE} and a straightforward computation based on \eqref{limitLfinalsepclass} show that the right-hand side of \eqref{QvarMvarphitilde} is indeed the quadratic variation of $\tilde{M}_\varphi$.
\end{proof}

\subsection{Limit SPDE}\label{sec:limDA}

\subsubsection{Covariance}\label{sec:limDAcovariance}

\begin{proposition} Let $S$ be defined by \eqref{OpS}. The operator $S$ is symmetric, non-negative and trace-class on the space $L^2(\T^N;\R^N)$.
\label{prop:Qtraceclass}\end{proposition}

\begin{proof}[Proof of Proposition~\ref{prop:Qtraceclass}] It is clear that $S$ is symmetric. That $S$ is non-negative means $\<S\rho,\rho\>\geq 0$, where $\<\cdot,\cdot\>$ is the canonical scalar product on $L^2(\T^N;\R^N)$ given as the sum over $i\in\{1,\dotsc,N\}$ of the $L^2$-scalar product of the components. By Lemma~\ref{lem:sympos}, we have, for all $\rho\in L^2(\T^N;\R^N)$,
$$
\<S\rho,\rho\>=\lim_{\delta\to 0}\delta\E\left|\int_0^\infty e^{-\delta t}\<\rho,\bar{E}_t\>dt\right|^2\geq 0,
$$
which shows that $\<S\rho,\rho\>\geq 0$ indeed. Let us prove that $S$ is trace-class. We fix an arbitrary orthonormal basis $(\zeta_k)$ of $L^2(\T^N;\R^N)$. For all $i,x$, we have $H(i,x,\cdot)\in L^2(\T^N;\R^N)$, where $H$ defined by \eqref{defKernelQ} is the kernel of $Q$ . We can use therefore the orthonormal decomposition
\begin{equation}\label{decompositionH}
H(i,x,\cdot)=\sum_k\<H(i,x,\cdot),\zeta_k\>\zeta_k=\sum_k S\zeta_k(i,x)\zeta_k.
\end{equation}
We evaluate this expansion at $(i,x)$, sum over $(i,x)$ and use the fact that $S$ is non-negative to obtain the classical identity that expresses $\mathrm{Trace}(S)$ has the sum over the set $\{1,\dotsc,N\}\times\T^N$ of the diagonal part $(i,x)\mapsto H(i,x,i,x)$. The bounds \eqref{BallR}, \eqref{RoE} then imply that $\mathrm{Trace}(S)\leq N\mathtt{R}$ is finite.
\end{proof}

To define the square-root of $S$ we employ the usual functional calculus for symmetric compact operators, based on the spectral decomposition. We have \begin{equation}\label{Sspectral}
S=\sum_{k\in\N}\lambda_k\zeta_k\otimes\zeta_k,
\end{equation}
where $(\lambda_k,\zeta_k)$ are the spectral elements of $S$ and $\zeta\otimes\zeta'$ is the notation for the rank-one operator that maps $\rho$ on  $\<\rho,\zeta\>\zeta'$. The square-root $S^{1/2}$ of $S$ is defined by the formula
\begin{equation}\label{Shalf}
S^{1/2}=\sum_{k\in\N}\lambda_k^{1/2}\zeta_k\otimes\zeta_k.
\end{equation}

We establish now the following result. 
\begin{proposition} The sum
\[
\sum_k \lambda_k\|\zeta_k\|_{H^{m}(\T^N;\R^N)}^2
\]
is finite.
\label{prop:phiphi}\end{proposition}

We will use Proposition~\ref{prop:phiphi} in the proof of Theorem~\ref{th:sollimeqEU}. A direct consequence of Proposition~\ref{prop:phiphi} is also that we can extend $S^{1/2}$ as a bounded operator $H^{-m}(\T^N;\R^N)\to L^2(\T^N;\R^N)$:
\begin{equation}\label{extShalf}
\|S^{1/2}z\|_{L^2(\T^N;\R^N)}\leq C\|z\|_{H^{-m}(\T^N;\R^N)}.
\end{equation} 
This extension is used in Proposition~\ref{prop:Phiwelldef} in particular.

\begin{proof}[Proof of Proposition~\ref{prop:phiphi}] Recall that $F=H^{2m}(\T^N;\R^N)$ has the standard Sobolev norm defined by \eqref{normF}. Let $\alpha$ be a multi-index of length $|\alpha|\leq m$. We integrate the identity $\lambda_k\zeta_k=S\zeta_k$ against $(-1)^{|\alpha|}\partial^{2\alpha}\zeta_k$ and integrate by parts to obtain the identity
\begin{equation}\label{phiphi1}
\lambda_k\|\partial^\alpha\zeta_k\|^2_{L^2(\T^N;\R^N)}=(-1)^{|\alpha|}\<\partial^{2\alpha}_x H,\zeta_k\otimes\zeta_k\>.
\end{equation}
Note that the procedure is valid because $H\in H^{2m}(\T^N\times\T^N;\R^N\times\R^N)$, a regularity property due to \eqref{defKernelQ} and \eqref{RoE}. Using \eqref{defKernelQ}, we have also the identity
\[
\<\partial^{2\alpha}_x H,\zeta_k\otimes\zeta_k\>=\E\<\partial^{2\alpha}R_0(\bar{E}_0),\zeta_k\>\<\bar{E}_0,\zeta_k\>.
\]
By the Parseval identity and the Cauchy-Schwarz inequality, it follows that
\[
\sum_k\lambda_k\|\partial^\alpha\zeta_k\|^2_{L^2(\T^N;\R^N)}\leq \E\left[\|\partial^{2\alpha}R_0(\bar{E}_0)\|_{L^2(\T^N;\R^N)}\|\bar{E}_0\|_{L^2(\T^N;\R^N)}\right],
\]
which is bounded by $\mathtt{R}^2$, owing to \eqref{BallR} and \eqref{RoE}. This concludes the proof.
\end{proof}
\subsubsection{Representation formula}\label{sec:RepresentationFormula}

Let $(\tilde{\rho}_t)$ be the process considered in Section~\ref{sec:cvDA}, defined as the a.s. limit of $(\tilde{\rho}^\eps_t)$. For $\rho\in H^{-1}(\T^N)$, $v\in L^2(\T^N;\R^N)$ let us set 
\begin{equation}\label{defPhit}
\Phi(\rho)v=\sqrt{2}\div_x(\rho S^{1/2}v).
\end{equation}

\begin{proposition} Let $s>2+N$. For $t\in[0,T]$, the application $t\mapsto\Phi(\tilde{\rho}_t)$ is well defined as a map from $U:=L^2(\T^N;\R^N)$ into $H:=H^{-s}(\T^N)$ and is a.s. continuous from $[0,T]$ into $L_2(U;H)$, the set of Hilbert-Schmidt operators from $U$ to $H$. Moreover, the process $t\mapsto\Phi(\tilde{\rho}_t)$ is adapted for the filtration $(\tilde{\F}_t)$ generated by $(\tilde{\rho}_t)$. 
\label{prop:Phiwelldef}\end{proposition}

\begin{proof}[Proof of Proposition~\ref{prop:Phiwelldef}] For smooth $v$ and $\rho$ defined on $\T^N$, we have
\begin{equation}\label{Phiwelldef1}
|\dual{\div_x(\rho S^{1/2}v)}{\xi}|=|\dual{v}{S^{1/2}(\rho\nabla_x\xi)}|\leq C\|v\|_{L^2(\T^N;\R^N)}\|\rho\nabla_x\xi\|_{H^{-m}(\T^N;\R^N)},
\end{equation}
where the estimate from above in \eqref{Phiwelldef1} is deduced from \eqref{extShalf}. The norm of the product $\rho\nabla_x\xi$ is bounded as follows: 
\[
\|\rho\nabla_x\xi\|_{H^{-m}(\T^N;\R^N)}\leq \|\rho\|_{H^{-1}(\T^N)}\|\xi\|_{C^2(\T^N)}.
\]
Let $s_1\in(2+N/2,s-N/2)$. Using the Sobolev injection of $H^{s_1}(\T^N)$ into $C^2(\T^N)$, we get the first bound 
\[
\|\Phi(\tilde{\rho}_t)\|_{L(U;H^{-s_1}(\T^N))}\leq C\|\tilde{\rho}_t\|_{H^{-1}(\T^N)}.
\]
Then we use the fact that the injection $H^{-s_1}(\T^N)\hookrightarrow H^{-s}(\T^N)=H$ is Hilbert-Schmidt, to obtain the desired estimate
\[
\|\Phi(\tilde{\rho}_t)\|_{L_2(U;H)}\leq C\|\tilde{\rho}_t\|_{H^{-1}(\T^N)}.
\]
Taking into account the almost sure continuity of $t\mapsto\tilde{\rho}_t$ from $[0,T]$ into $H^{-1}(\T^N)$, it is easy to conclude the proof.
\end{proof}

Note that $t\mapsto\Phi(\tilde{\rho}_t)$ is a predictable $L(U;H)$-valued process (because the process is adapted and has left-continuous trajectories). 

\begin{proposition} Let $(\tilde{X}_t)$ be the continuous $H$-valued martingale defined by 
\begin{equation}\label{deftildeX}
\tilde{X}_t=\tilde{\rho}_t-\tilde{\rho}_\mathrm{in}-\int_0^t \div_x(K_\sharp\nabla_x\tilde{\rho}_s+\Theta\tilde{\rho}_s)ds.
\end{equation} 
There exists a filtered probability space $(\hat{\Omega},\hat{\F},\hat{\P},(\hat{\F}_t))$, a $L^2(\T^N;\R^N)$-valued cylindrical Wiener process $W$ defined on $(\tilde{\Omega}\times\hat{\Omega},\tilde{\F}\times\hat{\F},\tilde{\P}\times\hat{\P})$, adapted to $(\tilde{\F}_t\times\hat{\F}_t)_t$, such that
\begin{equation}\label{repXtilde}
\tilde{X}_t(\tilde{\omega},\hat{\omega})=\int_0^t \Phi(\tilde{\rho}_s,\tilde{\omega},\hat{\omega})dW(s,\tilde{\omega},\hat{\omega}),
\end{equation}
where
\[
\tilde{X}_t(\tilde{\omega},\hat{\omega})=\tilde{X}_t(\tilde{\omega}),\quad \Phi(\tilde{\rho}_s,\tilde{\omega},\hat{\omega})=\Phi(\tilde{\rho}_s,\tilde{\omega}),
\]
for $\tilde{\P}\times\hat{\P}$-a.e. $(\tilde{\omega},\hat{\omega})\in\tilde{\Omega}\times\hat{\Omega}$.
\label{prop:QvartildeX}\end{proposition}

\begin{proof}[Proof of Proposition~\ref{prop:QvartildeX}] We apply Theorem~8.2 p.~220 in \cite{DaPratoZabczyk92}, with $Q$ the identity of $U$. The representation of $\tilde{X}_t$ as the stochastic integral \eqref{repXtilde} is then a consequence of the identity
\begin{equation}\label{XtildeQvarOK}
\<\tilde{X},\tilde{X}\>_t=\int_0^t\Phi(\tilde{\rho}_s)\Phi(\tilde{\rho}_s)^*ds,
\end{equation}
giving the quadratic variation of $(\tilde{X}_t)$. It is clear that, as claimed above, $\tilde{X}_t$ takes values in $H=H^{-s}(\T^N)$. Actually, $\tilde{\rho}(t)$ being in $H^{-1}(\T^N)$, $\tilde{X}_t$ even takes values in $H^{-3}(\T^N)$. For $\xi\in H^3(\T^N)$, we have
$\dual{\tilde{X}_t}{\xi}=\tilde{M}_{\varphi_\xi}(t)$, where $\varphi_\xi(\rho)=\dual{\rho}{\xi}$ and $\tilde{M}_\varphi$ is defined by \eqref{Mvarphitilde}. The quadratic variation of the $H$-valued martingale $(\tilde{X}_t)$ is defined as
\[
\<\tilde{X},\tilde{X}\>_t=\sum_{k,l}\<\tilde{X}_k,\tilde{X}_l\>_t\xi_k\otimes\xi_l,
\] 
where $(\xi_k)$ is an orthonormal basis of $H$ and $\tilde{X}_k(t)=\dual{\tilde{X}(t)\xi_k}{\xi_k}$. The formula \eqref{MtildeQvarOK} is true, therefore, if, and only if, for all $\xi\in H^s(\T^N)$, the real-valued martingale $(\tilde{M}_{\varphi_\xi}(t))$ has the quadratic variation
\begin{equation}\label{MtildeQvarOK}
\<\tilde{M}_{\varphi_\xi},\tilde{M}_{\varphi_\xi}\>_t=\int_0^t\|\Phi(\tilde{\rho}_s)^*\xi\|_{L^2(\T^N;\R^N)}^2ds.
\end{equation}
The quadratic variation of $\tilde{M}_{\varphi_\xi}$ is given by the formula \eqref{QvarMvarphitilde} with $\psi(s)=s$. To conclude, we simply need to observe that, by definition of $S$ and of $Q(\rho;\xi)$ in \eqref{defQrhoxi}, we have 
\[
Q(\rho;\xi)=\dual{S(\rho\nabla_x\xi)}{\rho\nabla_x\xi}_{L^2(\T^N;\R^N)}=\|S^{1/2}(\rho\nabla_x\xi)\|_{L^2(\T^N;\R^N)}^2,
\]
and thus $Q(\tilde{\rho}_s;\xi)=\|\Phi(\tilde{\rho}_s)^*\xi\|_{L^2(\T^N;\R^N)}^2$.
\end{proof}

We gather the results of Section~\ref{sec:cvDA} and Proposition~\ref{prop:QvartildeX} to give the following theorem. It is essentially the consequence of a slight abuse of notations, denoting by $(\tilde{\Omega},\tilde{\F},\tilde{\P})$ the whole probability space $(\tilde{\Omega}\times\hat{\Omega},\tilde{\F}\times\hat{\F},\tilde{\P}\times\hat{\P})$.

\begin{theorem}\label{th:tildesolSPDE} Under the hypotheses of Theorem~\ref{th:AD}, let $\eps_\N=\{\eps_n;n\in\N\}$, where $(\eps_n)$ is a sequence decreasing to $0$. There is a subset of $\eps_\N$ still denoted by $\eps_\N$, a filtered probability space $(\tilde{\Omega},\tilde{\F},\tilde{\P},(\tilde{\F}_t))$, some random variables $\{\tilde\rho^{\eps};\eps\in\eps_\N\}$, $\tilde{\rho}$ on $C([0,T];H^{-1}(\T^N))$, a $L^2(\T^N;\R^N)$-valued cylindrical\footnote{when we do not specify the covariance of the Wiener process, it is understood that it is the identity} Wiener process $\tilde{W}$ defined on $(\tilde{\Omega},\tilde{\F},\tilde{\P},(\tilde{\F}_t))$ such that:
\begin{enumerate}
\item for all $\eps\in\eps_\N$, the laws of $\rho^{\eps}$ and $\tilde\rho^{\eps}$ as $C([0,T];H^{-1}(\T^N))$-random variables coincide,
\item $\tilde{\P}$-a.s., $(\tilde\rho^{\eps})$ is converging to $\tilde\rho$ in $C([0,T];H^{-1}(\T^N))$ along $\eps_\N$,
\item\label{item3:martingale} the $H^{-1}(\T^d)$ process $\tilde\rho$ is $(\tilde{\F}_t)$-predictable, $\sup_{t\in[0,T]}\|\tilde{\rho}_t\|_{L^1(\T^N)}\leq\|\rho_\mathrm{in}\|_{L^1(\T^N)}$ a.s., and the following equality (in $H^{-s}(\T^N)$, $s>2+N/2$) is satisfied:
\begin{equation}\label{final formeqtilderho}
\tilde{\rho}_t=
\rho_\mathrm{in}+\int_0^t \div_x(K_\sharp\nabla_x\tilde{\rho}_s+\Theta\tilde{\rho}_s)ds
+\int_0^t \Phi(\tilde{\rho}_s)d\tilde{W}(s),
\end{equation}
for all $t\in[0,T]$, almost surely, where $\Phi(s)$ is defined by \eqref{defPhit}.
\end{enumerate}
\end{theorem}

Theorem~\ref{th:tildesolSPDE} states that, up to subsequence, $(\rho^\eps)_{\eps\in\eps_\N}$ is converging in law in the space  $C([0,T];H^{-1}(\T^N)$ to a weak-$L^1$ martingale solution to Equation~\eqref{eq:rho} with initial datum $\rho_\mathrm{in}$. This notion of ``weak-$L^1$ martingale solution'' is defined in the following section.

\subsubsection{Limit equation}\label{sec:limDAlim}

\begin{definition} Let $\rho_\mathrm{in}\in L^1(\T^N)$. A weak-$L^1$ martingale solution to Equation~\eqref{eq:rho} with initial datum $\rho_\mathrm{in}$ is a multiplet
\[
(\tilde{\Omega},\tilde{\F},\tilde{\P},(\tilde{\F}_t),\tilde{W},(\tilde{\rho}_t)),
\]
where $(\tilde{\Omega},\tilde{\F},\tilde{\P},(\tilde{\F}_t))$ is a filtered probability space, $\tilde{W}$ is a $L^2(\T^N;\R^N)$-valued cylindrical defined on $(\tilde{\Omega},\tilde{\F},\tilde{\P},(\tilde{\F}_t))$, $(\tilde{\rho}_t)$ is a process satisfying the properties given in item~\ref{item3:martingale} of Theorem~\ref{th:tildesolSPDE}.
\label{def:sollimeq}\end{definition}

\begin{theorem} Let $\rho_\mathrm{in}\in L^1(\T^N)$. Two weak-$L^1$ martingale solutions to \eqref{eq:rho} that have the same initial datum $\rho_\mathrm{in}$ and are constructed on the same stochastic basis coincide a.s.
\label{th:sollimeqEU}\end{theorem}

To establish this result of pathwise uniqueness for \eqref{eq:rho}, we will use the following result.

\begin{lemma} Let $K_\sharp$ be defined by \eqref{Ksharp}. Let $(\lambda_k,\zeta_k)$ denote the spectral elements of $S$ (see Section~\ref{sec:limDAcovariance}) and let $\varphi_k=\lambda_k^{1/2}\zeta_k$. For all $x\in\T^N$, the inequality
\begin{equation}\label{KsharpSuperParabolic}
K_\sharp(x)\geq K+\sum_k\varphi_k(x)\otimes\varphi_k(x)
\end{equation}
is satisfied in the sense of symmetric matrices. 
\label{lem:KsharpUpara}\end{lemma}

\begin{proof}[Proof of Lemma~\ref{lem:KsharpUpara}] To establish \eqref{KsharpSuperParabolic}, we use first \eqref{Ksharp} and \eqref{defKernelQ}, which gives
\[
K_\sharp(x)\geq K+(H(i,x,j,x))_{ij},
\]
since $b\geq 1$, whereas the term in factor of $(b-1)$ in \eqref{Ksharp} is a non-negative sym\-me\-tric matrix by Lemma~\ref{lem:sympos}. Since $\varphi_k=\lambda_k^{1/2}\zeta_k$, the expansion \eqref{decompositionH} can be rewritten as $H(i,x,j,x)=\sum_k\left[\varphi_k(x)\right]_i\left[\varphi_k(x)\right]_j$. This gives the desired result.
\end{proof}

\begin{proof}[Proof of Theorem~\ref{th:sollimeqEU}] We are given two weak-$L^1$ martingale solutions to \eqref{eq:rho} both with initial datum $\rho_\mathrm{in}$, and associated to the same probabilistic data $(\tilde{\Omega},\tilde{\F},\tilde{\P},(\tilde{\F}_t),\tilde{W})$. For simplicity of notations, we get rid of the tildes in what follows. By linearity, it is sufficient to consider the case where $\rho_\mathrm{in}\equiv0$ is trivial. If \eqref{eq:rho} was deterministic, a possible approach to uniqueness would be to regularize the equation, with the help of the Yosida regularization of the operator $-\div(K_\sharp\nabla\cdot)$. In that way, and although $\rho_t$ has no space-derivatives a priori, one can deal with the commutators that appear when one tries to do an energy estimate for a regularization of $\rho\mapsto\|\rho\|_{L^1(\T^N)}$. This approach does not work for \eqref{eq:rho}, since there are actually two second-order operators at stake there: the second one appears when we write the It\^o correction to the martingale term. Instead of proving a renormalization property therefore, we will use a duality method. Let $t_*\in(0,T]$ be fixed, and let $\psi_*$ be a given $\F^W_{t_*}$-measurable function. We consider a solution $(\psi,Z)$ of the backward SPDE
\begin{equation}\label{eq:psiZ}
d\psi=\left[-\div_x(K_\sharp\nabla\psi)+\Theta\cdot\nabla\psi\right]dt-\sqrt{2}\varphi\cdot\nabla_x Z dt+Z\cdot dW(t),
\end{equation}
for $t\in(0,t_*)$, with terminal condition
\begin{equation}\label{eq:psiTC}
\psi(t_*)=\psi_*.
\end{equation}
Let us explain the notation used in \eqref{eq:psiZ} and what we mean by ``solution''  $(\psi,Z)$. The component $\varphi_k$ of $\varphi$ are defined in Lemma~\ref{lem:KsharpUpara}. Let $n=\left[\frac{N}{2}\right]+1$. Since $m>\frac{N}{2}+n$, Proposition~\ref{prop:phiphi} and the usual Sobolev's embedding show that $\varphi$ is an element of $\ell^2(\N;C^n(\T^N))$. The products $\varphi\cdot\nabla_x Z$ and $Z\cdot dW(t)$ stand for
\[
\sum_k \varphi_k\cdot\nabla_x Z_k,\quad \sum_k Z_k d\beta_k(t)
\]
respectively. Here, we use the decomposition (see Proposition~4.1 of \cite{DaPratoZabczyk92})
\[
S^{1/2}W(t)=\sum_k \varphi_k\beta_k(t),
\]
where $(\beta_1(t),\beta_2(t),\dotsc)$ is a family of independent one-dimensional Wiener processes. By $\mathcal{P}$ we denote the $\sigma$-algebra of predictable sets, based on the filtration $(\F_t)$. A couple $(\psi,Z)$ is said to be solution to \eqref{eq:psiZ}-\eqref{eq:psiTC} on $(0,t_*)$ if 
\begin{enumerate}
\item $\psi\in L^2(\Omega\times(0,t_*),\mathcal{P},H^2(\T^N))$, $Z\in  L^2(\Omega\times(0,t*),\mathcal{P},\ell^2(\N;H^1(\T^N)))$,
\item $\psi\in C([0,t_*];L^2(\T^N))$ almost surely,
\item for all $t\in[0,t_*]$, almost surely,
\begin{multline}\label{psiZBSPDE}
\psi(t,x)=\psi_*(x)+\int_t^{t_*}\left[\div_x(K_\sharp(x)\nabla\psi(s,x))-\Theta(x)\cdot\nabla\psi(s,x)\right]ds\\
+\sqrt{2}\int_t^{t_*}\varphi(x)\cdot\nabla_x Z(s,x)ds-\int_t^{t_*}Z(s)\cdot dW(s),
\end{multline}
for a.e. $x\in\T^N$.
\end{enumerate}
The equation \eqref{eq:psiZ} is super-parabolic in the sense of Assumption~2.2 of \cite{DuMeng2010}. This is an application of the estimate~\eqref{KsharpSuperParabolic}. By Theorem~2.2 in \cite{DuMeng2010}, a solution $(\psi,Z)$ to \eqref{eq:psiZ}-\eqref{eq:psiTC} as above does exist, provided $\psi_*\in L^2(\Omega,\F^W_{t_*},H^1(\T^N))$. Actually, Theorem~2.2 of \cite{DuMeng2010} applies in the case where $W$ is a finite-dimensional Wiener process. However, as asserted in Remark~2.3 of  \cite{DuMeng2010}, the result continues to hold in the case of the cylindrical Wiener process $W$ as considered here. This assertion must be specified a bit however. Indeed, recasting the condition~(2.4) of \cite{DuMeng2010} in our framework, we need a bound on the quantity
\begin{equation}\label{LipsigmaDuMeng}
\left[\sum_k \Lip(\varphi_k)^2\right]^{1/2}.
\end{equation}
As $\varphi\in\ell^2(\N;C^n(\T^N))$ (recall Proposition~\ref{prop:phiphi}), the quantity \eqref{LipsigmaDuMeng} is indeed finite. Similarly, using Theorem~2.3 of \cite{DuMeng2010} and the fact that $\varphi\in\ell^2(\N;C^n(\T^N))$, we get the higher differentiability property
\[
\psi\in L^2(\Omega\times(0,t_*),\mathcal{P},H^{n+2}(\T^N)),\quad Z\in  L^2(\Omega\times(0,t*),\mathcal{P},\ell^2(\N;H^{n+1}(\T^N))),
\]
provided $\psi_*\in L^2(\Omega,\F^W_{t_*},H^{n+1}(\T^N))$. Since $n>N/2$, this shows that $(\psi,Z)$ have respectively $C^2$ and $C^1$ regularity in $x$. In particular, the equation \eqref{psiZBSPDE} is satisfied pointwise, for every $x\in\T^N$. By subtracting \eqref{psiZBSPDE} written at $t=0$, we obtain
\begin{multline}\label{psiZBSPDEforward}
\psi(t,x)=\psi(0,x)-\int_0^{t}\left[\div_x(K_\sharp(x)\nabla\psi(s,x))-\Theta(x)\cdot\nabla\psi(s,x)\right]ds\\
-\sqrt{2}\int_0^{t}\varphi(x)\cdot\nabla_x Z(s,x)ds+\int_0^{t}Z(s)\cdot dW(s),
\end{multline}
for every $x\in\T^N$. Let $J_\delta$ be the regularizing operator defined by convolution with an approximation of the unit on $\T^N$. By testing \eqref{final formeqtilderho} with a function $J_\delta^*\xi$, we obtain the regularized equation
\begin{multline}\label{rhodelta}
\rho_\delta(t,x)=\rho_\delta(0,x)+\int_0^{t}\div_\delta(K_\sharp\nabla\rho(s,x))+\Theta(x)\rho(s,x))ds\\
+\sqrt{2}\int_0^{t}\div_\delta(\rho(s,x)\varphi(x) \cdot dW(s)),
\end{multline}
where $\rho_\delta=J_\delta\rho$, $\div_\delta=J_\delta\div_x$. We apply the It\^o formula to the two diffusions \eqref{psiZBSPDEforward} and \eqref{rhodelta}, take expectancy and integrate the result with respect to $x$. Since $\rho_\delta(0)=0$, this gives the identity
\begin{multline}\label{preUdelta}
\E\<\rho_\delta(t_*),\psi_*\>=\E\int_0^{t^*}\<(\rho-\rho_\delta)(s),\div(K_\sharp\nabla\psi(s))-\Theta\cdot\nabla\psi(s)-\sqrt{2}\varphi\cdot\nabla Z(s)\>ds\\
+\E\int_0^{t^*}\<\rho(s),\div(K_\sharp(\nabla_\delta-\nabla)\psi(s))-\Theta\cdot(\nabla_\delta-\nabla)\psi(s)-\sqrt{2}\varphi\cdot(\nabla_\delta-\nabla) Z(s)\>ds,
\end{multline}
where we use the duality product between $L^1(\T^N)$ and $C(\T^N)$ and the notation $\nabla_\delta=\nabla J_\delta^*$. The regularity of $(\psi,Z)$ is sufficient to justify that, in the limit $\delta\to0$, \eqref{preUdelta} gives $\E\<\rho(t_*),\psi_*\>=0$. Since $\psi_*$ is arbitrary in the class $L^2(\Omega,\F^W_{t_*},H^{n+1}(\T^N))$, this implies $\rho(t_*)=0$ almost surely.
\end{proof}

\subsubsection{Conclusion}\label{sec:conclusion}

We use the Gy\"ongy-Krylov argument, \cite[Lemma~1.1]{GyongyKrylov96}. We deduce that \eqref{eq:rho} has a weak-$L^1$ solution, strong in the probabilistic sense: there does exist a weak-$L^1$ martingale solution with probabilistic data that coincides with a set of probabilistic data prescribed in advance. Moreover, weak-$L^1$ martingale solutions with given initial datum to \eqref{eq:rho} are unique. Consequently, the whole sequence $(\rho_\eps)$ considered in Theorem~\ref{th:tildesolSPDE} is converging in law to the weak-$L^1$ martingale solution to \eqref{eq:rho} with initial datum $\rho_\mathrm{in}$. This concludes the proof of Theorem~\ref{th:AD}.\bigskip


{\small This work was supported by the LABEX MILYON (ANR-10-LABX-0070) of Universit\'e de Lyon, within the program ``Investissements d'Avenir" (ANR-11-IDEX-0007) operated by the French National Research Agency (ANR), by the ANR project STAB and and by the LABEX Lebesgue Center of Mathematics, program ANR-11-LABX-0020-01. }

\appendix

\section{Resolution of the unperturbed equation}\label{app1}

Consider the $\LB$ case first. By integration with respect to $v$ in the equation
\begin{equation}\label{vStovfromsLB}
\partial_t f_t+E(t,s;\e)\cdot\nabla_v f_t+f_t=\rho(f_t)M,
\end{equation}
one checks that $\rho(f_t)=\rho(f)$ for all $t\geq 0$. Therefore, the formula \eqref{fLBEEstar} is simply the Duhamel formula associated to the PDE \eqref{vStovfromsLB}. In the $\FP$ case, instead of working on the PDE
\begin{equation}\label{vStovfromsFP}
\partial_t f_t+E(t,s;\e)\cdot\nabla_v f_t=\QFP f_t,
\end{equation}
we work on the solution $V_t$ to the equation
\begin{equation}\label{VtFP}
dV_t=(-V_t+E(t,s;\e))dt+\sqrt{2}dB_t,\quad t\geq s.
\end{equation}
If $V_s$ has the law of density $f$ with respect to the Lebesgue measure on $\R^N$, then by \eqref{ffromlaw} (with no dependence on $x$ here), we obtain, by explicit integration in \eqref{VtFP}, 
\begin{multline*}
\int_{\R^N}\varphi(v) f_{s,t}^{\FP}(v) dv\\
=\int_{\R^N}\int_{\R^N}\varphi\left(e^{-(t-s)}v+\int_s^t e^{-(t-\sigma)} E(\sigma,s;\e) ds+\sqrt{1-e^{-2(t-s)}}\, w\right)M(w)f(v)dw dv.
\end{multline*}
A change of variable gives \eqref{fFPEEstar} then.

\section{Martingale property of Markov processes}\label{app2}


In this section, we make the connection between a Markov process and the Martingale problem associated to its generator. Although this is a fundamental topic, we found complete references (of Formula~\eqref{VarXMartingaleE}, giving the expression of the quadratic variation in terms of the integral of the Carr\'e du Champ operator) only in the case of finite-dimensional state spaces. Theorem~\ref{th:XMartingaleE} is given for functions $\varphi\in\mathrm{BC}(E)$ (continuous and bounded functions). Some standard argument, using truncates, allow a generalization to Lipschitz functions, as long as the processes at stake have sufficient moments. This generalization of Theorem~\ref{th:XMartingaleE} is used in the proof of Proposition~\ref{prop:MartingalePbtilde} for instance.\medskip

Let $E$ be a Polish space. Let $(X_t)$ be an $E$-valued time-homogeneous Markov process with respect to a filtration $(\F_t)$, with Markov semi-group $(P_t)$. The generator $\L$ associated to $(P_t)$ is defined by means of the bounded pointwise convergence \cite{Priola99}. Let $\Delta_t=t^{-1}(P_t-\mathrm{Id})$. A function $\varphi$ of $\mathrm{BC}(E)$ is in $D(\L)$ if the family $(\Delta_t\varphi)_{0<t<1}$ is bounded for the norm $\|\varphi\|_{\mathrm{BC}(E)}=\sup_{x\in E}|\varphi(x)|$ and if there exists $\psi\in\mathrm{BC}(E)$ such that
\[
\Delta_t\varphi(x)\to\psi(x)
\]
when $t\to0^+$ for all $x\in E$. We set then $\L\varphi=\psi$.

\begin{theorem} Let $E$ be a Polish space. Let $(X_t)$ be an $E$-valued time-homogeneous Markov process with respect to a filtration $(\F_t)$, with Markov semi-group $(P_t)$ of generator $\L$: for all $\varphi\in\mathrm{BC}(E)$
\begin{equation}\label{defMarkov}
\E\left[\varphi(X_{t+s})|\F_t\right]=(P_s\varphi)(X_t).
\end{equation} 
Assume that $t\mapsto P_t\varphi(x)$ is continuous, for all $\varphi\in\mathrm{BC}(E)$, $x\in E$. Assume that $(\omega,t)\mapsto X_t(\omega)$ is measurable $\Omega\times\R_+\to E$. Then, for all $\varphi$ in the domain of $\L$,
\begin{equation}\label{XMartingaleE}
M_\varphi(t):=\varphi(X_t)-\varphi(X_0)-\int_0^t\L\varphi(X_s)ds
\end{equation}
is a $(\F_t)$-martingale. Assume furthermore that 
$|\varphi|^2$ is in the domain of $\L$. Then the process $(Z_t)$ defined by 
\begin{equation}\label{VarXMartingaleE}
Z_t:=|M_\varphi(t)|^2-\int_0^t (\L|\varphi|^2-2\varphi\L\varphi)(X_s)ds,
\end{equation}
is a martingale.
\label{th:XMartingaleE}\end{theorem}

\begin{remark}\label{rk:Qvarcadlag} Assume that $(X_t)$ is c\`adl\`ag. Then the process
\begin{equation}\label{defAt}
A_t:=\int_0^t (\L|\varphi|^2-2\varphi\L\varphi)(X_s)ds
\end{equation}
is continuous and adapted, and thus predictable. Consequently, $(A_t)$ is the predictable quadratic variation $\<M_\varphi,M_\varphi\>_t$, \cite[p.38]{JacodShiryaev03}, of $M_\varphi$: this is the compensator, \cite[p.32]{JacodShiryaev03}, of the quadratic variation $[M_\varphi,M_\varphi]_t$, \cite[p.51]{JacodShiryaev03}, of $M_\varphi$.
\end{remark}

Note that we assume also continuity from the left of $t\mapsto P_t\varphi(x)$ in Theorem~\ref{th:XMartingaleE}. If $\varphi\in D(\L)$, this ensures that $t\mapsto P_t\varphi(x)$ is differentiable, with $\frac{d\;}{dt}P_t\varphi(x)=P_t\L\varphi(x)$, \cite[Proposition~3.2]{Priola99}.

\begin{proof}[Proof of Theorem~\ref{th:XMartingaleE}] Let $0\leq s\leq t$. By the Markov property~\eqref{defMarkov}, we have
\begin{align*}
\E[M_\varphi(t)|\F_s]-M_\varphi(s)&=\E[M_\varphi(t)-M_\varphi(s)|\F_s]\\
&=P_{t-s}\varphi(X_s)-\varphi(X_s)-\int_s^t [P_{\sigma-s}\L\varphi](X_s) d\sigma.
\end{align*}
We use the relation $\frac{d\;}{dt}P_t\varphi(x)=P_t\L\varphi(x)$ to obtain the martingale property. Indeed, this gives
$$
P_{t-s}\varphi-\varphi=\int_s^t P_{\sigma-s}\L\varphi  d\sigma,
$$
and thus $\E[M_\varphi(t)|\F_s]-M_\varphi(s)=0$. The proof of the martingale property for \eqref{VarXMartingaleE} is divided in several steps. By $C(\varphi)$, we will denote any constant that depend on $\varphi$ and may vary from lines to lines. We fix a subdivision $\sigma=(t_i)_{0,n}$ of $[0,T]$.
In a first step, we show that
\begin{equation}\label{MtMAtincr}
A_t=\lim_{|\sigma|\to 0}\sum_{i=0}^{n-1}\E\left[A_{t\wedge t_{i+1}}-A_{t\wedge t_i}|\F_{t_i}\right],
\end{equation}
with a convergence in $L^2(\Omega)$. Indeed, we have
\begin{equation}\label{MtMAtincr2}
A_t=\sum_{i=0}^{n-1} A_{t\wedge t_{i+1}}-A_{t\wedge t_i},
\end{equation}
and $\zeta(t_{i+1}):=A_{t\wedge t_{i+1}}-A_{t\wedge t_i}-\E\left[A_{t\wedge t_{i+1}}-A_{t\wedge t_i}|\F_{t_i}\right]$ satisfies
\begin{equation}\label{MtMpropzeta}
\E[\zeta(t_i)\zeta(t_j)]=0,\quad i\not=j,\quad |\zeta(t_{i+1})|\leq C(\varphi)(t_{i+1}-t_i),
\end{equation}
where $C(\varphi)=\|\L\varphi^2\|_{\mathrm{BC}(E)}+2\|\varphi\|_{\mathrm{BC}(E)}\|\L\varphi\|_{\mathrm{BC}(E)}$. It follows that
\[
\E\left|\sum_{i=0}^{n-1}\zeta(t_{i+1})\right|^2=\E\sum_{i=0}^{n-1}\left|\zeta(t_{i+1})\right|^2\leq C(\varphi)T |\sigma|,
\]
which tends to $0$ when $|\sigma|\to 0$. Using \eqref{MtMAtincr2}, we obtain \eqref{MtMAtincr}.
In a second step we prove that 
\begin{equation}\label{MtMStep2a}
|M_\varphi(t_{i+1})-M_\varphi(t_i)|^2=|\varphi(X_{t_{i+1}})-\varphi(X_{t_i})|^2+R_{t_i,t_{i+1}},
\end{equation}
with
\begin{equation}\label{MtMStep2b}
\E\sum_{i=0}^{n-1}|R_{t_i,t_{i+1}}|=\mathcal{O}(|\sigma|^{1/2}).
\end{equation}
By definition of $M_\varphi(t)$, \eqref{MtMStep2a} is satisfied with a remainder term
\begin{equation}\label{MMRii}
R_{t_i,t_{i+1}}=\left|\int_{t_i}^{t_{i+1}} \L\varphi(X_s) ds\right|^2-2(\varphi(X_{t_{i+1}})-\varphi(X_{t_i}))\int_{t_i}^{t_{i+1}} \L\varphi(X_s) ds.
\end{equation}
Using the fact that $\varphi^2\in D(\L)$, we have also
\begin{multline*}
|\varphi(X_{t_{i+1}})-\varphi(X_{t_i})|^2
=M_{\varphi^2}(t_{i+1})-M_{\varphi^2}(t_i)-2\varphi(X_{t_i})(M_{\varphi}(t_{i+1})-M_{\varphi}(t_i))\\
+\int_{t_i}^{t_{i+1}}\L\varphi^2(X_s)ds-2\varphi(X_{t_i})\int_{t_i}^{t_{i+1}}\L\varphi(X_s)ds.
\end{multline*}
It follows that
\begin{equation}
\E[|\varphi(X_{t_{i+1}})-\varphi(X_{t_i})|^2|\F_{t_i}]
=\int_{t_i}^{t_{i+1}}\E\left[\left(\L\varphi^2(X_s)-2\varphi(X_{t_i})\L\varphi(X_s)\right)|\F_{t_i}\right] ds.\label{MMdelta}
\end{equation}
Taking expectation in \eqref{MMdelta}, we get the following bound.
\begin{equation}\label{MMdeltaphi}
\E[|\varphi(X_{t_{i+1}})-\varphi(X_{t_i})|^2]\leq C_\varphi (t_{i+1}-t_i).
\end{equation}
Consider now the cross-product term in the right-hand side of \eqref{MMRii}. Using Young's inequality with a parameter $\eta>0$, we see that the term $\E|R_{t_i,t_{i+1}}|$ can be bounded by 
$$
(1+\eta^{-1})\E\left|\int_{t_i}^{t_{i+1}} \L\varphi(X_s) ds\right|^2+\eta \E[|\varphi(X_{t_{i+1}})-\varphi(X_{t_i})|^2],
$$
and thus, taking $\eta=(t_{i+1}-t_i)^{1/2},$ bounded from above by $C(\varphi)(t_{i+1}-t_i)^{3/2}$. This gives \eqref{MtMStep2b}.
The third step establishes the limit
\begin{equation}\label{ModifiedQuadVar}
A_t=\lim_{|\sigma|\to 0}\sum_{i=0}^{n-1}\E\left[|M_\varphi(t_{i+1})-M_\varphi(t_i)|^2|\F_{t_i}\right],
\end{equation}
with a convergence in $L^1(\Omega)$. 
To that purpose, we note that \eqref{MtMStep2a} shows that we can replace the increment $M_\varphi(t_{i+1})-M_\varphi(t_i)$ by the increment $\varphi(t_{i+1})-\varphi(t_i)$ in the right-hand side of \eqref{ModifiedQuadVar}. This gives an error term $\eps_1(|\sigma|)$ which converges to $0$ in $L^1(\Omega)$, taking \eqref{MtMStep2b} into account. By \eqref{MtMAtincr} and\eqref{MMdelta}, we deduce that
\begin{equation}\label{MtMpre1}
A_t-\sum_{i=0}^{n-1}\E\left[|M_\varphi(t_{i+1})-M_\varphi(t_i)|^2|\F_{t_i}\right]=\eps_2(|\sigma|)+r(t,\sigma),
\end{equation}
where  $\eps_2(|\sigma|)$ converges to $0$ in $L^1(\Omega)$ and
\[
|r(t,\sigma)|\leq 2\sum_{i=0}^{n-1}\int_{t_i}^{t_{i+1}}\left|(\varphi(X_{t_i})-\varphi(X_s))\L\varphi(X_s)\right|ds.
\]
We have in particular
\[
|r(t,\sigma)|\leq C(\varphi)\sum_{i=0}^{n-1}\int_{t_i}^{t_{i+1}}\left|\varphi(X_{t_i})-\varphi(X_s)\right|ds
\]
and an estimate similar to \eqref{MMdeltaphi} (obtained by working on the increment $\varphi(X_s)-\varphi(X_{t_i})$ instead of $\varphi(X_{t_{i+1}})-\varphi(X_{t_i})$) shows that 
\begin{equation}\label{phiXcontquadmean}
\E|\varphi(X_s)-\varphi(X_{t_i})|^2\leq C(\varphi)(s-t_i).
\end{equation} 
We deduce that $r(t,\sigma)$ is converging to $0$ in $L^2(\Omega)$ when $|\sigma|\to 0$.
At last, let us show that $Z_t=|M_\varphi(t)|^2-A_t$ is a martingale. Let $0\leq s<t$. Set $t_{n+1}=\min\{t_i;t_i\geq t\}$, $t_{l+1}=\min\{t_i;t_i\geq s\}$. We may assume $t_n\geq s$. Then $\E[Z_t-Z_s|\F_s]$ is the limit when $|\sigma|\to 0$ of the quantity 
\begin{equation}\label{Zmartingalea}
\E\Big[|M_\varphi(t)|^2-|M_\varphi(s)|^2-\sum_{i=l}^{n-1}\E\left[|M_\varphi(t_{i+1})-M_\varphi(t_i)|^2|\F_{t_i}\right]\Big|\F_s\Big]
\end{equation}
By the tower property $\E[\E[Y|\F_{t_i}]|\F_s]=\E[Y|\F_s]$ if $t_i\geq s$, and the usual cancellation properties for martingales, \eqref{Zmartingalea} is equal to
\begin{equation}\label{Zmartingaleb}
\E\Big[|M_\varphi(t)-M_\varphi(t_n)|^2+\E[|M_\varphi(s)-M_\varphi(t_l)|^2|\F_{t_l}]\Big|\F_s\Big].
\end{equation}
Using \eqref{phiXcontquadmean}, we see that \eqref{Zmartingaleb} tends to zero in $L^1(\Omega)$. This gives the desired result
$\E[Z_t-Z_s|\F_s]=0$.
\end{proof}


\def\cprime{$'$}

\end{document}